\documentclass{ws-jml}
\usepackage[utf8]{inputenc}
\usepackage[english]{babel}

\newcommand{\pp}{\mathsf p}
\newcommand{\cc}{\mathsf c}
\newcommand{\II}{\mathsf I}
\newcommand{\NN}{\mathbb N}

\renewcommand{\ss}{\mathsf s}

\newcommand{\cb}{BF}
\newcommand{\cpp}{PV}
\newcommand{\cefp}{DV}
\newcommand{\cek}{Cm}
\newcommand{\ceu}{Cs}
\newcommand{\cec}{Cn}
\newcommand{\csmn}{SMN}
\newcommand{\cfp}{DV}

\newcommand{\sx}{{\overline x}}
\newcommand{\sy}{{\overline y}}

\newcommand{\st}{{\overline t}}
\newcommand{\sz}{{\overline z}}
\newcommand{\su}{{\overline u}}
\newcommand{\sv}{{\overline v}}
\newcommand{\sr}{{\overline r}}
\newcommand{\sk}{{\overline k}}
\newcommand{\sm}{{\overline m}}

\newcommand{\zz}{\mathsf z}
\newcommand{\xx}{\mathsf x}

\newcommand{\kk}{\mathsf k}

\newcommand{\rvf}{\mathrel{\mathbf{r}^V_f}}

\newcommand{\rvfx}[1]{\mathrel{\mathbf{r}^V_{f,\,#1}}}

\newcommand{\BQC}{\ensuremath{\mathsf{BQC}}}

\begin{document}

\markboth{Aleksandr Yu. Konovalov}
{Generalized Realizability and Basic Logic}

\catchline{}{}{}{}{}

\title{Generalized Realizability and Basic Logic}
\author{Aleksandr Yu. Konovalov}
\address{
Faculty of Mechanics and Mathematics, \\
Lomonosov Moscow State University, \\
GSP-1, Leninskie Gory, Moscow, 119991, Russian Federation, \\
alexandr.konoval@gmail.com}

\maketitle


\begin{abstract}
Let $V$ be a set of number-theoretical functions.
We define a notion of absolute $V$-realizability for predicate formulas and sequents in such a way
that the indices of functions in $V$ are used for interpreting the implication and the universal quantifier.
In this paper we prove that
Basic Logic is sound with respect to the semantics of absolute $V$-realizability
if $V$ satisfies some natural conditions.
\end{abstract}



\section{Introduction}

The notion of recursive realizability was introduced by S.~C.~Kleene \cite{klini1945}.
It specifies the informal intuitionistic semantics by partial recursive functions \cite{plisko_review}.
A natural generalization of recursive realizability is the $V$-realizability for
some set of functions $V$, where functions from the set~$V$ are used instead of partial recursive functions.
Recently, special cases of $V$-realizability were considered:
primitive recursive realizability \cite{dam1994,sal2001_1},
minimal realizability \cite{dam1995},
arithmetical realizability \cite{kon_ar_bl,kon_ar_pr},
hyperarithmetical realizability \cite{kon_plisko_hr}.
Intuitionistic Logic is sound with respect to the semantics of recursive realizability. 
But in general this is not the case for the $V$-realizability
\cite{kon_ar_bl,kon_plisko_hr,viter_dis,pak_dis}.
Basic Logic was introduced in \cite{visser,ruitenburg}.
It is weaker than Intuitionistic Logic.
For example, the formula $(\top \to P) \to P$ is not derivable in Basic Logic.
The aim of this paper is to prove that Basic Logic is sound with respect to the semantics
of $V$-realizability if $V$ satisfies some natural conditions.

\section{Definitions}
\subsection{$V$-functions}
We begin with some notation.
Denote by $\NN$ the set of all natural numbers $0, 1, 2, \ldots$
Let $\cc$ be a bijection of $\NN^2$ to $\NN$. 
Denote by $\pp_1, \pp_2$ the functions of $\NN$ to $\NN$
such that, for all $a, b \in \NN$, $\pp_1(\cc(a, b)) = a$ and $\pp_2(\cc(a, b)) = b$.
We omit the brackets in expressions of the form $\pp_1(t'),\ \pp_2(t'')$ and write $\pp_1 t',\ \pp_2 t''$.
Suppose $n \ge 1$ and $1 \le i \le n$, denote by $I^i_n$ the function of $\NN^n$ to $\NN$ such that $I^i_n(a_1, \ldots, a_n) = a_i$
for all $a_1, \ldots, a_n \in \NN$.

We consider an arbitrary (countable) set $V$ of partial functions with arguments and values from $\NN$.
We say that $\varphi$ is a $V$-function if $\varphi\in V$.
For every $n \ge 0$, denote by $V_n$ the set of all $n$-ary $V$-functions.
Clearly, $V = \bigcup^\infty_{n = 0} V_n$.
For every $n \ge 0$, let us fix some numbering of the set $V_n$.
This means that we fix some set of indices $\II_n \subseteq \NN$ and a mapping $e \mapsto \varphi^{V,\,n}_e$ such that
$\varphi^{V,\,n}_e$ is an $n$-ary $V$-function whenever $e \in \II_n$ and
every $n$-ary $V$-function is $\varphi^{V,\,n}_e$ for some $e \in \II_n$.
We often write $\varphi^V_e$ instead of $\varphi^{V,\,n}_e$ if there is no confusion.

Let $Var = \{x_1, x_2, \ldots\}$ be a countable set of variables.
We say that an expression $t$ is a \textit{$V$-term} if
$t$ is a natural number or
$t \in Var$ or
$t$ has the form $\varphi(t_1, \ldots, t_n)$, 
where $\varphi \in V_n$ and $t_1, \ldots, t_n$ are $V$-terms, for some $n \ge 0$.
Any $V$-term without variables is called \textit{closed}.
Suppose $e$ is a natural number and $t$ is a closed $V$-term,
then the relation ``$e$ is the value of $t$'' is defined inductively by the length of $t$:
$e$ is the value of $t$ if $t$ is the natural number $e$;
$e$ is the value of $\varphi(t_1, \ldots, t_n)$ if there are natural numbers $e_1, \dots, e_n$ such that
$e_1, \dots, e_n$ are the values of $t_1, \ldots, t_n$,
$\varphi(e_1, \ldots, e_n)$ is defined, and $e = \varphi(e_1, \ldots, e_n)$.
We say that the value of a closed $V$-term $t$ is defined
if there is a natural number $e$ such that $e$ is the value of $t$.
It can be easily checked that if the value of closed $V$-term $t$ is defined,
then there exists a unique natural number $e$ such that $e$ is the value of $t$.
In this case we denote by $\overline{t}$ the value of $t$.
Suppose $t_1, t_2$ are closed $V$-terms,
we write $t_1 \simeq t_2$
if either (i) the values of $t_1$ and $t_2$ are not defined,
or (ii) the values of $t_1$ and $t_2$ are defined
and $\overline{t}_1 = \overline{t}_2$.
Let $k_1, \ldots, k_n$ be natural numbers, $x_1, \ldots, x_n$ distinct variables, and $t$ an $V$-term,
denote by $[k_1, \ldots, k_n/x_1, \ldots, x_n]\,t$
the result of substituting $k_1, \ldots, k_n$ for all occurrences of $x_1, \ldots, x_n$ in $t$.
Suppose $t_1,\ t_2$ are $V$-terms
and all variables in $t_1$ and $t_2$ are in a list of distinct variables $x_1, \ldots, x_n$,
we write $t_1 \simeq t_2$
if for all natural numbers $k_1, \ldots, k_n$ we have
$[k_1, \ldots, k_n/x_1, \ldots, x_n]\,t_1 \simeq [k_1, \ldots, k_n/x_1, \ldots, x_n]\,t_2.$

We assume that the following conditions hold:
\begin{itemize}
  \item[(\cb)]
  $I^i_n$, $\cc$, $\pp_1$, $\pp_2$ are $V$-functions for all $n \ge 1$,\ $1 \le i \le n$;

  \item[(\cek)]
  the composition of $V$-functions is a $V$-function and
  an index of it can be obtained by some $V$-function:
  for all natural numbers $n, m_1, \ldots, m_n$ there is an $(n+1)$-ary $V$-function $s$ such that
  $s(e, e_1, \ldots, e_n) \in \II_m$ and
\begin{equation*}
  \varphi^V_{s(e, e_1, \ldots, e_n)}(x_1, \ldots, x_m) \simeq
  \varphi^V_e(\varphi^V_{e_1}(x_1, \ldots, x_{m_1}), \ldots, \varphi^V_{e_n}(x_1, \ldots, x_{m_n}))
\end{equation*}
for all~$e \in \II_n, e_1 \in \II_{m_1}, \ldots, e_n \in \II_{m_n}$,
where $m = \max_{1 \le i \le n}{m_i}$;

\item[(\cec)]
  every constant function is a $V$-function and
  an index of it can be obtained by some $V$-function:
  there exists a $V$-function $s$ such that, for all natural numbers $k$,
  we have $s(k) \in \II_0$ and $\varphi^{V,\ 0}_{s(k)} \simeq k$.

\item[(\ceu)]
an index of a ``conditional function'' can be obtained by some $V$-function:
for every natural number $n$ there is a $V$-function $s$
such that, for all natural numbers $d$ and $e_1, e_2 \in \II_{n+1}$, we have $s(e_1, e_2) \in \II_{n+1}$,
\begin{align*}
\varphi^V_{s(e_1, e_2)}(x_1, \ldots, x_n, d) \simeq \varphi^V_{e_1}(x_1, \ldots, x_n, d)\ \mbox{ if } \pp_1 d = 0, \\
 \varphi^V_{s(e_1, e_2)}(x_1, \ldots, x_n, d) \simeq \varphi^V_{e_2}(x_1, \ldots, x_n, d)\ \mbox{ if } \pp_1 d \not= 0;
\end{align*}
\end{itemize}

For example,
if $\cc$, $\pp_1$, $\pp_2$ are recursive (see \S 5.3 in \cite{rodgers}), then
the following sets of functions with some numbering satisfy the conditions
(\cb), (\cek), (\cec), (\ceu):
\begin{itemize}
  \item the set of all partial recursive functions;
  \item the set of all arithmetical functions (see \cite{kon_ar_bl,kon_ar_pr});
  \item the set of all hyperarithmetical functions (see \cite{kon_plisko_hr});
  \item the set of all $L$-defined functions, where $L$ is an extension of the language of arithmetic
        (see \cite{kon_Lr_IPC,kon_Lr_IPC_IS}).
\end{itemize}

Now we show that the following conditions hold:
\begin{itemize}
  \item[(\cpp)]
  Any permutation of variables is available for the $V$-functions:
  if $p$ is a permutation of the set $\{1, \ldots, n\}$,
  then there is a $V$-function $s$ such that,
  for all $e \in \II_n$, $s(e) \in \II_n$ and
  $
  \varphi^V_{s(e)}(x_1, \ldots, x_n) \simeq \varphi^V_e(x_{p(1)}, \ldots, x_{p(n)});
  $


  \item[(\cefp)] Adding of a dummy variable is available for the $V$-functions:
  for all natural numbers $n$ there exists a $V$-function $s$ such that,
    for all $e \in \II_n$, $s(e) \in \II_{n+1}$ and
    $
    \varphi^V_{s(e)}(x_1, \ldots, x_n, x_{n+1}) \simeq \varphi^V_e(x_1, \ldots, x_n);
    $
  \item[(\csmn)] An analog of the ($s-m-n$)-theorem (Theorem V \S 1.8 in \cite{rodgers}) holds:
  for all natural numbers $m,\ n$ there exists a $V$-function $s$ such that,
for all natural numbers $k_1, \ldots, k_m$ and $e \in \II_{m+n}$, we have $s(e, k_1, \ldots, k_m) \in \II_n$ and
\begin{equation*}
  \varphi^V_{s(e, k_1, \ldots, k_m)}(x_1, \ldots, x_n) \simeq \varphi^V_e(x_1, \ldots, x_n, k_1, \ldots, k_m).
\end{equation*}
\end{itemize}

\begin{lemma}
    (\cb), (\cek), (\cec) imply (\cpp).
\end{lemma}
\begin{proof}
Let $p$ be a permutation of the set $\{1, \ldots, n\}$.
Since $x_{p(j)} \simeq I^{p(j)}_n(x_1, \ldots, x_n)$ for all $j = 1, \ldots, n$,
we see that, for all $e \in \II_n$,
\begin{equation*}
\varphi^V_e(x_{p(1)}, \ldots, x_{p(n)})
\simeq
  \varphi^V_e(I^{p(1)}_n(x_1, \ldots, x_n), \ldots, I^{p(n)}_n(x_1, \ldots, x_n)).
\end{equation*}
It follows from (\cb) that there are natural numbers $i_1, \ldots, i_n$ such that
$i_j$ is an index of $I^{p(j)}_n$ for all $j = 1, \ldots, n$.
Using (\cek), we get that there exists a $V$-function $s'$ such that, for all $e \in \II_n$,
\begin{equation*}
\varphi^V_{s'(e, i_1, \ldots, i_n)}(x_1, \ldots, x_n)
\simeq
  \varphi^V_e(I^{p(1)}_n(x_1, \ldots, x_n), \ldots, I^{p(n)}_n(x_1, \ldots, x_n)).
\end{equation*}
Thus for all $e \in \II_n$ we have
\begin{equation}\label{l1_eq_1}
\varphi^V_{s'(e, i_1, \ldots, i_n)}(x_1, \ldots, x_n)
\simeq
\varphi^V_e(x_{p(1)}, \ldots, x_{p(n)}).
\end{equation}
By (\cec),  there are natural numbers $l_1, \ldots, l_n$ such that
$\varphi^V_{l_j} \simeq i_j$ for all $j = 1, \ldots, n$.
Let $i$ denote an index of $I^1_1$.
It is obvious that, for all $e \in \II_n$,
\begin{equation*}
  s'(e, i_1, \ldots, i_n) \simeq s'(\varphi^V_i(e), \varphi^V_{l_1}, \ldots, \varphi^V_{l_n}).
\end{equation*}
It follows from (\cek) that there exists a $V$-function $s$ such that
\begin{equation*}
  s(x) \simeq s'(\varphi^V_i(x), \varphi^V_{l_1}, \ldots, \varphi^V_{l_n}).
\end{equation*}
Thus for all natural numbers $e$ we have
\begin{equation}\label{l1_eq_2}
  s(e) \simeq s'(e, i_1, \ldots, i_n).
\end{equation}
From \eqref{l1_eq_1}, \eqref{l1_eq_2} it follows that, for all $e \in \II_n$,
$$\varphi^V_{s(e)}(x_1, \ldots, x_n) \simeq \varphi^V_e(x_{p(1)}, \ldots, x_{p(n)}).$$

\end{proof}

\begin{lemma}
    (\cb), (\cek), (\cec) imply (\cefp).
\end{lemma}
\begin{proof}
By (\cb),  there are natural numbers $i_1, \ldots, i_n$ such that
$i_j$ is an index of $I^j_{n+1}$ for all $j = 1, \ldots, n$.
It is obvious that, for all $e \in \II_n$,
\begin{equation*}
  \varphi^V_e(x_1, \ldots, x_n)
  \simeq \varphi^V_e(I^1_{n+1}(x_1, \ldots, x_n, x_{n+1}), \ldots, I^n_{n+1}(x_1, \ldots, x_n, x_{n+1})).
\end{equation*}
It follows from (\cek) that there exists a $V$-function $s'$ such that, for all $e \in \II_n$,
$s'(e, i_1, \ldots, i_n) \in \II_{n+1}$ and
\begin{equation*}
  \varphi^V_{s'(e, i_1, \ldots, i_n)}(x_1, \ldots, x_n, x_{n+1})
  \simeq \varphi^V_e(I^1_{n+1}(x_1, \ldots, x_n, x_{n+1}), \ldots, I^n_{n+1}(x_1, \ldots, x_n, x_{n+1})).
\end{equation*}
Thus for all $e \in \II_n$ we have
\begin{equation}\label{l2_eq_1}
  \varphi^V_{s'(e, i_1, \ldots, i_n)}(x_1, \ldots, x_n, x_{n+1})
  \simeq \varphi^V_e(x_1, \ldots, x_n).
\end{equation}
By (\cec),  there are natural numbers $l_1, \ldots, l_n$ such that
$\varphi^V_{l_j} \simeq i_j$ for all $j = 1, \ldots, n$.
Let $i$ denote an index of $I^1_1$.
It is obvious that, for all $e \in \II_n$,
\begin{equation*}
  s'(e, i_1, \ldots, i_n) \simeq s'(\varphi^V_i(e), \varphi^V_{l_1}, \ldots, \varphi^V_{l_n}).
\end{equation*}
It follows from (\cek) that there exists a $V$-function $s$ such that
\begin{equation*}
  s(x) \simeq s'(\varphi^V_i(x), \varphi^V_{l_1}, \ldots, \varphi^V_{l_n}).
\end{equation*}
Thus for all natural numbers $e$ we have
\begin{equation}\label{l2_eq_2}
  s(e) \simeq s'(e, i_1, \ldots, i_n).
\end{equation}
From \eqref{l2_eq_1}, \eqref{l2_eq_2} it follows that
$
  \varphi^V_{s(e)}(x_1, \ldots, x_n, x_{n+1})
  \simeq \varphi^V_e(x_1, \ldots, x_n)
$
for all $e \in \II_n$.
\end{proof}

\begin{lemma}
    (\cb), (\cek), (\cec) imply (\csmn).
\end{lemma}
\begin{proof}
By (\cec),
there is a $V$-function $s'$ such that,
for every $k$, we have $s'(k) \in \II_0$ and $\varphi^V_{s'(k)} \simeq k$.
Obviously, for all natural numbers $k_1, \ldots, k_m$ and $e \in \II_{n+m}$,
\begin{equation*}
   \varphi^V_e(I^1_n(\sx), \ldots, I^n_n(\sx), \varphi^V_{s'(k_1)}, \ldots, \varphi^V_{s'(k_m)})
   \simeq \varphi^V_e(\sx, k_1, \ldots, k_m),
\end{equation*}
where $\sx = x_1, \ldots, x_n$.
It follows from (\cb) that there are natural numbers $i_1, \ldots, i_n$ such that
$i_j$ is an index of $I^j_n$ for all $j = 1, \ldots, n$.
It follows from (\cek) that there exists a $V$-function $s''$ such that,
for all natural numbers $k_1, \ldots, k_m$ and $e \in \II_{n+m}$,
$s''(e, i_1, \ldots, i_n, s'(k_1), \ldots, s'(k_m)) \in \II_n$
and
\begin{equation*}
  \varphi^V_{s''(e, i_1, \ldots, i_n, s'(k_1), \ldots, s'(k_m))}(\sx)
  \simeq
  \varphi^V_e(I^1_n(\sx), \ldots, I^n_n(\sx), \varphi^V_{s'(k_1)}, \ldots, \varphi^V_{s'(k_m)}),
\end{equation*}
where $\sx = x_1, \ldots, x_n$.
Thus for all natural numbers $k_1, \ldots, k_m$ and $e \in \II_{n+m}$,
\begin{equation}\label{l3_eq_1}
  \varphi^V_{s''(e, i_1, \ldots, i_n, s'(k_1), \ldots, s'(k_m))}(x_1, \ldots, x_n)
  \simeq \varphi^V_e(x_1, \ldots, x_n, k_1, \ldots, k_m)
\end{equation}
For each $j = 1, \ldots, m$ denote by $\psi_j$ the function such that
$\psi_j(x, \sy) \simeq s'(I^{j+1}_{m+1}(x, \sy)),$
where $\sy = y_1, \ldots, y_m$.
It follows from (\cb), (\cek) that  $\psi_j$ is a $V$-function
for all $j = 1, \ldots, m$.
By (\cec), there are natural numbers $l_1, \ldots, l_n$ such that
$\varphi^V_{l_j} \simeq i_j$ for all $j = 1, \ldots, n$.
It follows from (\cek) that there exists a $V$-function $s$ such that
$$
  s(x, \sy) \simeq s''(I^1_{n+1}(x, \sy), \varphi^V_{l_1}, \ldots, \varphi^V_{l_n}, \psi_1(x, \sy), \ldots, \psi_m(x, \sy)),
$$
where $\sy = y_1, \ldots, y_m$.
Since $\psi_j(x, y_1, \ldots, y_m) \simeq s'(y_j)$ for all $j = 1, \ldots, m$
and $\varphi^V_{l_j} \simeq i_j$ for all $j = 1, \ldots, n$,
we see that
\begin{equation}\label{l3_eq_2}
  s(x, y_1, \ldots, y_m) \simeq s''(x, i_1, \ldots, i_n, s'(y_1), \ldots, s'(y_m)).
\end{equation}
From \eqref{l3_eq_1}, \eqref{l3_eq_2} it follows that
$$
\varphi^V_{s(e, k_1, \ldots, k_m)}(x_1, \ldots, x_n) \simeq \varphi^V_e(x_1, \ldots, x_n, k_1, \ldots, k_m)
$$
for all $e \in \II_n$.

\end{proof}

\subsection{Basic Predicate Calculus}
Basic Predicate Calculus ($\BQC$) is introduced in \cite{ruitenburg}.

\textit{The language of $\BQC$} contains
a countably infinite set of predicate symbols for each finite arity,
a countably infinite set of variables,
parentheses,
the logical constants $\bot$ (falsehood), $\top$ (truth),
the logical connectives $\land$, $\lor$, $\to$
and the quantifiers $\forall$, $\exists$.
Suppose $M \subseteq \NN$, denote by $L^M_\BQC$
the extension of the language of $\BQC$ by individual constants from the set $M$.
Thus the language of $\BQC$ is a special case of $L^M_\BQC$ for $M = \varnothing$.
We write $L_\BQC$ instead of $L^\varnothing_\BQC$.

\textit{Terms} of $L^M_\BQC$ are constants from $M$ and variables.
\textit{Atoms} of $L^M_\BQC$ are $\bot,\ \top$, and expressions of the form $P(t_1, \ldots, t_n)$,
where $P$ is an $n$-ary predicate symbol
and $t_1, \ldots, t_n$ are terms of $L^M_\BQC$.
\textit{Formulas} of $L^M_\BQC$ are built up 
according to the following grammar:
\begin{equation*}\label{grammarLBQC}
  A,\, B ::= \Phi \mid A \land B \mid A \lor B \mid \forall \sx\,(A \to B) \mid \exists y\, A;
\end{equation*}
here
$\Phi$ is an atom of $L^M_\BQC$,
$\sx$ is a (possibly empty) list of distinct variables,
and $y$ is a variable.
We write $A \to B$ instead of $\forall\,(A \to B)$.
Terms and formulas of $L^M_\BQC$ will be called \textit{$M$-terms} and \textit{$M$-formulas}, for short.
At the same time formulas of $L_\BQC$ are said to be \textit{formulas}.


Free and bound variables are defined in the usual way.
An occurrence of a variable $x$ in an $M$-formula $A$ is \textit{free}
if it
is not in the scope of a quantifier $\exists x$ or $\forall \sz$ in $A$, where $x$ is in $\sz$.
An occurrence of a variable in an $M$-formula that is not free is called \textit{bound}.
We say that a variable $x$ is a \textit{free variable} (\textit{bound variable}) of an $M$-formula $A$
if there exists a free (bound) occurrence of $x$ in $A$.
A sentence of $L^M_\BQC$ is a formula of $L^M_\BQC$ without free variables.
Sentences of $L^M_\BQC$ are called \textit{$M$-sentences}, and sentences of $L_\BQC$ simply \textit{sentences}, for short.

An $M$-term $t$ is called \textit{free} for a variable $x$ in a $M$-formula $A$
if for each variable $y$ in $t$ there is no occurrence of $x$
in the scope of a quantifier $\exists y$ or 
$\forall \sz$ for some $\sz$ such that $y$ is in $\sz$.
Let $t_1, \ldots, t_n$ be $M$-terms, $x_1, \ldots, x_n$ be distinct variables,
and $A$ be an $M$-formula,
denote by $[t_1, \ldots, t_n/x_1, \ldots, x_n] A$
the result of substituting $t_1, \ldots, t_n$ for all free occurrences of $x_1, \ldots, x_n$
in a formula $A'$
obtained from $A$ by renaming all bound variables
in such a way that, for each $i = 1, \ldots, n$, the $M$-term $t_i$ is free for $x_i$ in $A'$.

Suppose $A$ is an $M$-formula
and all free variables of $A$ are in $\sx$,
where $\sx$ is a list of distinct variables.
By the statement ``$A(\sx)$ is a $M$-formula''
we mean the conjunction of statements: ``$A$ is an $M$-formula'',
``$\sx$ is a list of distinct variables'',
and ``all free variables of $A$ are in $\sx$''.

If $\st = t_1, \ldots, t_n$ is a list of $M$-terms,
then put $|\st| \rightleftharpoons n$.
Let $A(\sx)$ be an $M$-formula
and $\st$ be a list of $M$-terms such that $|\st| = |\sx|$;
then by $A(\st)$ denote $[\st/ \sx] A$.



A \textit{sequent} is an expression of the form $A \Rightarrow B$, where $A$ and $B$ are formulas.


The axioms of $\BQC$ are:
\smallskip

A1) $A \Rightarrow A$;

A2) $A \Rightarrow \top$;

A3) $\bot \Rightarrow A$;

A4) $A\land \exists x\,B \Rightarrow \exists x\,(A\land  B)$, where $x$ is not free in $A$;

A5) $A\land (B\lor C) \Rightarrow (A\land  B)\lor (A\land  C)$;

A6) $\forall \overline x\,(A\to B)\land  \forall \overline x\,(B\to C) \Rightarrow \forall\overline x\,(A\to C)$;

A7) $\forall \overline x\,(A\to B)\land  \forall \overline x\,(A\to C) \Rightarrow \forall \overline x\,(A\to B\land  C)$;

A8) $\forall \overline x\,(B\to A)\land \forall \overline x\,(C\to A) \Rightarrow \forall \overline x\,(B\lor C\to A)$;

A9) $\forall \overline x\,(A \to B) \Rightarrow \forall \overline x\,([\sy/\sx] A \to [\sy/\sx] B)$;

A10) $\forall \overline x (A \to B) \Rightarrow\forall \sy (A \to B)$,
where no variable in $\sy$ is free in $\forall \overline x (A \to B)$;

A11) $\forall \overline x,x\,(B\to A) \Rightarrow \forall \overline x\,(\exists x\, B\to A)$, where $x$ is not free in $A$.
\medskip

The rules of $\BQC$ are:
\medskip

R1) $\frac{\displaystyle A\Rightarrow B\; B\Rightarrow C}{\displaystyle A\Rightarrow C}$;
\medskip

R2) $\frac{\displaystyle A\Rightarrow B\; A\Rightarrow C}{\displaystyle A\Rightarrow B\land  C}$;
\medskip

R3) $\frac{\displaystyle A\Rightarrow B\land  C}{\displaystyle A\Rightarrow B}$ (a),\;
$\frac{\displaystyle A\Rightarrow B\land  C}{\displaystyle A\Rightarrow C}$ (b);
\medskip

R4) $\frac{\displaystyle B\Rightarrow A\; C\Rightarrow A}{\displaystyle B\lor C\Rightarrow A}$;
\medskip

R5) $\frac{\displaystyle B\lor C\Rightarrow A}{\displaystyle B\Rightarrow A}$ (a),\;
$\frac{\displaystyle B\lor C\Rightarrow A}{\displaystyle C\Rightarrow A}$ (b);
\medskip

R6) $\frac{\displaystyle A \Rightarrow B}{\displaystyle [\sy/\sx] A \to [\sy/\sx] B}$;
\medskip

R7) $\frac{\displaystyle B\Rightarrow A}{\displaystyle \exists x\, B\Rightarrow A}$, where $x$ is not free in $A$;
\medskip

R8) $\frac{\displaystyle \exists x\,B\Rightarrow A}{\displaystyle B\Rightarrow A}$, where $x$ is not free in $A$;
\medskip

R9) $\frac{\displaystyle A\land  B\Rightarrow C}{\displaystyle A\Rightarrow \forall \overline x(B\to C)}$,
where each variable in $\sx$ is not free in $A$.
\medskip

In the axioms and rules of $\BQC$\ $A,\ B,\ C$ are formulas, 
$\sx$ and $\sy$ are lists of distinct variables such that $|\sx| = |\sy|$,
and $x$ is a variable.

Given a sequent $S$,
we write $\BQC \vdash S$ if $S$ is derivable in $\BQC$.
We say that a formula $A$ is derivable in $\BQC$ if $\BQC \vdash \top \Rightarrow A$.

\subsection{$V$-realizability}


In \cite{kon_gr_for_lang_ar,kon_gr_for_lang_ar_IS} we introduced a notion of $V$-realizability for the language of arithmetic.
Using methods of \cite{plisko1983,kon_MP,kon_MP_IS},
in this paper
we define a notion of absolute $V$-realizability in some domain $M \subseteq \NN$
for the formulas of $L_\BQC$.

Suppose $M \subseteq \NN$,
we call any total function from $M^n$ to $2^\NN$ an \textit{$n$-ary generalized predicate} on $M$,
where $2^\NN$ is the set of all subsets of $\NN$.
A mapping $f$ is called an \textit{$M$-evaluation} if
$f(P)$ is an $n$-ary generalized predicate on $M$
whenever $P$ is an $n$-ary predicate symbol of $L_\BQC$.
We write $P^f$ instead of $f(P)$.
We say that $f$ is an \textit{evaluation}
if $f$ is an $M$-evaluation for some $M \subseteq \NN$.

\begin{definition}
Let $e$ be a natural number, $M$ a subset of $\NN$, $f$ an $M$-evaluation, and $A$ an $M$-sentence.
The relation  ‘‘$e$ $V$-\textit{realizes} $A$ \textit{on} $f$'' is denoted
$e \rvf A$ and
is defined by induction on the number of logical connectives
and quantifiers in $A$:
\begin{itemize}
\item there is no $e$ such that $e \rvf \bot$;
\smallskip

\item $e \rvf \top$ for all $e$;
\smallskip

\item $e \rvf P(a_1, \ldots, a_n) \rightleftharpoons e \in P^f(a_1, \ldots, a_n)$,
where $P$ is an $n$-ary predicate symbol
and $a_1, \ldots, a_n \in M$;
\smallskip

\item $e \rvf (\Phi \land \Psi) \rightleftharpoons$ $\pp_1 e \rvf \Phi$ and $\pp_2 e \rvf \Psi$;
\smallskip

\item $e \rvf (\Phi \lor \Psi) \rightleftharpoons$
$(\pp_1 e = 0$ and $\pp_2 e \rvf \Phi)$
or $(\pp_1 e = 1$ and $\pp_2 e \rvf \Psi)$;
\smallskip

\item $e \rvf \exists x \:\Phi(x) \rightleftharpoons \pp_1 e \in M$ and $\pp_2 e \rvf \Phi(\pp_1 e)$;
\smallskip

\item $e \rvf \forall x_1, \ldots, \forall x_n\,(\Phi(x_1, \ldots, x_n) \to \Psi(x_1, \ldots, x_n)) \rightleftharpoons$
$e \in \II_{n+1}$  and, for all
$s \in \NN$,\ $a_1, \ldots, a_n \in M$,
if $s \rvf \Phi(a_1, \ldots, a_n)$, then
$\varphi^V_e(a_1, \ldots, a_n, s)$ is defined
and
$\varphi^V_e(a_1, \ldots, a_n, s)\rvf \Psi(a_1, \ldots, a_n)$.

\end{itemize}
\end{definition}
A sentence $A$ is called \textit{absolutely $V$-realizable over all domains}
if there exists a natural number $e$ such that, for all $M \subseteq \NN$, we have $e \rvf A$ whenever $f$ is an $M$-evaluation.

We say that a list of distinct variables $\sx$ is \textit{admissible} for a sequent $A \Rightarrow B$
if all free variables of the formulas $A$ and $B$ are in $\sx$.
By definition, put
\begin{equation*}
    e \rvfx{\sx} A \Rightarrow B \rightleftharpoons
    e \rvf \forall \sx\:(A \to B);
  \end{equation*}
here $e$ is a natural number,
$f$ is an evaluation,
$A \Rightarrow B$ is a sequent,
and
$\sx$ is an admissible list of variables for $A \Rightarrow B$.

\begin{lemma}\label{l_rvfx_tr}
Let $A \Rightarrow B$ be a sequent,
$x_1, \ldots, x_n$ an admissible list of variables for $A \Rightarrow B$,
and $p$ a permutation of $\{1, \ldots, n\}$.
For all $e \in \II_{n+1}$ there exists $e' \in \II_{n+1}$ such that,
for every evaluation~$f$,
$e \rvfx{x_{p(1)}, \ldots, x_{p(n)}} A \Rightarrow B$
iff $e' \rvfx{x_1, \ldots, x_n} A \Rightarrow B$.
\end{lemma}
\begin{proof}
It follows from (\cpp) that, for all $e \in \II_{n+1}$, there exists $e' \in \II_{n+1}$
such that
\begin{equation*}
    \varphi^V_{e'}(k_1, \ldots, k_n, a) \simeq \varphi^V_e(k_{p(1)}, \ldots, k_{p(n)}, a)
\end{equation*}
for all natural numbers $k_1, \ldots, k_n, a$.
It can be easily checked that, for every evaluation~$f$, we have
$e \rvfx{x_{p(1)}, \ldots, x_{p(n)}} A \Rightarrow B$
if and only if $e' \rvfx{x_1, \ldots, x_n} A \Rightarrow B$.
\end{proof}

\begin{lemma}\label{l_rvfx_ficvar}
  Let $A \Rightarrow B$ be a sequent,
  $z_1, \ldots, z_n$ an admissible list of variables for $A \Rightarrow B$,
  and $u_1, \ldots, u_m$ a list of variables such that
  the list $z_1, \ldots, z_n, u_1, \ldots, u_m$ is admissible for $A \Rightarrow B$.
  For all $e \in \II_{n+1}$ there exists $e' \in \II_{n+m+1}$ such that,
  for every evaluation~$f$,
  $e \rvfx{z_1, \ldots, z_n} A \Rightarrow B$
  iff $e' \rvfx{z_1, \ldots, z_n, u_1, \ldots, u_m} A \Rightarrow B$.
\end{lemma}
\begin{proof}
It follows from (\cfp), (\cpp) that, for every $e \in \II_{n+1}$, there exists $e' \in \II_{n+m+1}$
such that, for all natural numbers $k_1, \ldots, k_{m+n}, a$, we have
\begin{equation*}
    \varphi^V_{e'}(k_1, \ldots, k_n, k_{n+1}, \ldots, k_{m+n}, a) \simeq \varphi^V_e(k_1, \ldots, k_n, a).
\end{equation*}
It can be easily checked that,
for every evaluation~$f$, we have
  $e \rvfx{z_1, \ldots, z_n} A \Rightarrow B$
  if and only if $e' \rvfx{z_1, \ldots, z_n, u_1, \ldots, u_m} A \Rightarrow B$.
\end{proof}

\begin{lemma}\label{l_rvfx_ficvar_del}
  Under the conditions of Lemma \ref{l_rvfx_ficvar},
  for all $e' \in \II_{n+m+1}$ there exists $e \in \II_{n+1}$ such that,
  for every evaluation~$f$,
  $e' \rvfx{z_1, \ldots, z_n, u_1, \ldots, u_m} A \Rightarrow B$
  if and only if $e \rvfx{z_1, \ldots, z_n} A \Rightarrow B$.
\end{lemma}
\begin{proof}
It follows from (\cpp), (\csmn) that for all $e' \in \II_{n+m+1}$ there exists $e \in \II_{n+1}$
such that, for all natural numbers $k_1, \ldots, k_n, a$,
\begin{equation*}
    \varphi^V_e(k_1, \ldots, k_n, a) \simeq \varphi^V_{e'}(k_1, \ldots, k_n, 0, \ldots, 0, a).
\end{equation*}
It can be easily checked that,
for every evaluation~$f$,
$e' \rvfx{z_1, \ldots, z_n, u_1, \ldots, u_m} A \Rightarrow B$
if and only if $e \rvfx{z_1, \ldots, z_n} A \Rightarrow B$.
\end{proof}

\begin{proposition}\label{p_eq}
 Let $S$ be a sequent,
 $\sx$ and $\sy$ admissible lists of variables for $S$,
 $|\sx| = n$ and $|\sy| = m$.
 For all $e \in \II_{n+1}$ there exists $e' \in \II_{m+1}$ such that,
 for every evaluation~$f$,\
 $e' \rvfx{\sx} S$
 if and only if $e \rvfx{\sy} S$.
\end{proposition}
\begin{proof}
  Denote by $\sz$ a list of distinct variables such that, for every variable $w$, we have
  $w$ in $\sz$ if and only if $w$ in $\sx$ and $w$ in $\sy$.
  Note that $\sz$ is admissible for $S$.
  Let $\su$ be a list of distinct variables such that,
  for every variable $w$, we have $w$ in $\su$ if and only if $w$ in $\sx$ and $w$ is not in $\sy$.
  Denote by $\sv$ a list of distinct variables such that,
  for every variable $w$, we have $w$ in $\sv$ if and only if $w$ in $\sy$ and $w$ is not in $\sx$.
  Let $e \in \II_{n+1}$.
  It follows from
  Lemmas \ref{l_rvfx_tr}, \ref{l_rvfx_ficvar}, \ref{l_rvfx_ficvar_del} that
  there are natural numbers $e_1,\ e_2,\ e_3,\ e'$ such that
  \begin{equation*}
  e \rvfx{\sx} S \Longleftrightarrow
  e_1 \rvfx{\sz, \su} S \Longleftrightarrow
  e_2 \rvfx{\sz} S \Longleftrightarrow
  e_3 \rvfx{\sz, \sv} S \Longleftrightarrow
  e' \rvfx{\sy} S
 \end{equation*}
 for all evaluations $f$.
\end{proof}

\section{Main result}

Our main result is the following.
\begin{theorem}\label{t_main}
  If a sequent $S$ is derivable in $\BQC$
  and $\sr = r_1, \ldots, r_l$ is an admissible list of variables for $S$,
  then there exists a natural number $e$ such that $e \rvfx{\sr} S$ for all evaluations $f$.
\end{theorem}
\begin{proof}
  By induction on derivations of $S$.
  Suppose $S$ is an axiom of $\BQC$.
  \begin{itemize}
\item[A1)] Let $S$ be $A(\sr) \Rightarrow A(\sr).$
By (\cb)
there is a natural number $e$ such that, for all natural numbers $k_1, \ldots, k_l, d$, we have
\begin{equation}\label{fta1f}
  \varphi^V_e(k_1, \ldots, k_l, d) \simeq d.
\end{equation}
Let $\varnothing \not= M \subseteq \NN$ and $f$ be an $M$-evaluation.
Suppose
\begin{equation}\label{fta1g}
  d \rvf A(k_1, \ldots, k_l)
\end{equation}
for some natural numbers $d$ and $k_1, \ldots, k_l \in M$.
From \eqref{fta1f}, \eqref{fta1g} it follows that
\begin{equation}\label{fta1e}
  \varphi^V_e(k_1, \ldots, k_l, d) \rvf A(k_1, \ldots, k_l).
\end{equation}
Thus for all natural numbers $d$ and $k_1, \ldots, k_l \in M$ we have \eqref{fta1e}
whenever \eqref{fta1g}. Hence
$
  e \rvfx{\sr} A(\sr) \Rightarrow A(\sr).
$

  \item[A2)] Let $S$ be
$A(\sr) \Rightarrow \top.$
By (\cb)
 there is a natural number $e$ such that, for all natural numbers $k_1, \ldots, k_l, d$, we have \eqref{fta1f}.
Let $f$ be an evaluation.
It can be easily checked that
$
  e \rvfx{\sr} A(\sr) \Rightarrow \top.
$

  \item[A3)] Let $S$ be $\bot \Rightarrow A(\sr).$
It can be easily checked that, for every $e \in \II_{l+1}$, we have
$
  e \rvfx{\sr} \bot \Rightarrow A(\sr)
$
for all evaluations $f$.

  \item[A4)] Let $S$ be
$A(\sr) \land \exists x\, B(x, \sr) \Rightarrow \exists x\,(A(\sr) \land  B(x, \sr)).$
It follows from (\cb), (\cek), (\cefp), (\cpp) that
there is a natural number $e$ such that, for all natural numbers $k_1, \ldots, k_l, d$, we have
\begin{equation}\label{ft_a4_f}
  \varphi^V_e(k_1, \ldots, k_l, d) \simeq \cc(\pp_1\pp_2 d,\cc(\pp_1 d,\pp_2\pp_2 d)).
\end{equation}
Let $\varnothing \not= M \subseteq \NN$ and $f$ be an $M$-evaluation.
Suppose
\begin{equation}\label{ft_a4_b}
  d \rvf A(\sk) \land \exists x\,B(x, \sk)
\end{equation}
for some natural number $d$ and $\sk = k_1, \ldots, k_l \in M$.
Let us prove that
\begin{equation}\label{ft_a4_e}
  \varphi^V_e(k_1, \ldots, k_l, d) \rvf \exists x \:(A(\sk) \land B(x, \sk)).
\end{equation}
Using \eqref{ft_a4_b}, we get
\begin{equation}\label{ft_a4_b1}
  \pp_1 d \rvf A(\sk),
\end{equation}
\begin{equation}\label{ft_a4_b2}
  \pp_2 d \rvf \exists x\,B(x, \sk).
\end{equation}
From \eqref{ft_a4_b2} it follows that
\begin{equation}\label{ft_a4_b20}
  \pp_2\pp_2 d \rvf B(\pp_1\pp_2 d, \sk).
\end{equation}
Using \eqref{ft_a4_b1} and \eqref{ft_a4_b20}, we obtain
\begin{equation}\label{ft_a4_c1}
  \cc(\pp_1 d,\pp_2\pp_2 d) \rvf A(\sk) \land B(\pp_1\pp_2 d, \sk).
\end{equation}
From \eqref{ft_a4_c1} it follows that
\begin{equation}\label{ft_a4_cc}
  \cc(\pp_1\pp_2 d, \cc(\pp_1 d,\pp_2\pp_2 d)) \rvf \exists x \:(A(\sk) \land B(x, \sk)).
\end{equation}
Using \eqref{ft_a4_f} and \eqref{ft_a4_cc}, we obtain \eqref{ft_a4_e}.
Thus for all natural numbers $d$ and $k_1, \ldots, k_l \in M$
we have \eqref{ft_a4_e} whenever \eqref{ft_a4_b}.
Hence $e \rvfx{\sr} S$.

  \item[A5)] Let $S$ be
$A(\sr) \land (B(\sr) \lor C(\sr)) \Rightarrow (A(\sr) \land  B(\sr))\lor (A(\sr) \land  C(\sr)).$
By (\cb), (\cek), (\cefp), and (\cpp),
there is a natural number $e$ such that, for all natural numbers $k_1, \ldots, k_l, d$, we have
\begin{equation}\label{fta5f}
  \varphi^V_e(k_1, \ldots, k_l, d) \simeq \mathsf{c(p_1p_2}d, \mathsf{c(p_1}d, \mathsf{p_2p_2}d)).
\end{equation}
Let $\varnothing \not= M \subseteq \NN$ and $f$ be an $M$-evaluation.
Suppose
\begin{equation}\label{fta5r}
  d \rvf A(\sk) \land (B(\sk) \lor C(\sk))
\end{equation}
for some natural number $d$ and $\sk = k_1, \ldots, k_l \in M$.
Let us prove that
\begin{equation}\label{fta5e}
  \varphi^V_e(k_1, \ldots, k_l, d) \rvf (A(\sk) \land B(\sk)) \lor (A(\sk) \land C(\sk)).
\end{equation}
From \eqref{fta5r} it follows that
\begin{equation}\label{fta5r1}
  \pp_1 d \rvf A(\sk),
\end{equation}
\begin{equation}\label{fta5r2}
  \pp_2 d \rvf (B(\sk) \lor C(\sk)).
\end{equation}
Using \eqref{fta5r2}, we have
\begin{equation}\label{fta5r20}
  (\pp_1 \pp_2 d = 0 \text{ and } \pp_2 \pp_2 d \rvf B(\sk))
\text{ or } (\pp_1 \pp_2 d = 1 \text{ and } \pp_2 \pp_2 d \rvf C(\sk)).
\end{equation}
Using \eqref{fta5r1} and \eqref{fta5r20}, we obtain
\begin{equation}\label{fta5r20c1}
  \pp_1 \pp_2 d = 0 \land \cc(\pp_1 d, \pp_2 \pp_2 d) \rvf (A(\sk) \land B(\sk))
\end{equation}
or
\begin{equation}\label{fta5r20c2}
  \pp_1 \pp_2 d = 1 \land \cc(\pp_1 d, \pp_2 \pp_2 d) \rvf (A(\sk) \land C(\sk)).
\end{equation}
Hence
\begin{equation}\label{fta5r20cc}
  \mathsf{c(p_1p_2}d, \mathsf{c(p_1}d, \mathsf{p_2p_2}d))
  \rvf (A(\sk) \land B(\sk)) \lor (A(\sk) \land C(\sk)).
\end{equation}
Using \eqref{fta5f} and \eqref{fta5r20cc}, we obtan \eqref{fta5e}.
Thus for all natural numbers $d$ and $k_1, \ldots, k_l \in M$
we have \eqref{fta5e} whenever \eqref{fta5r}.
Hence $e \rvfx{\sr} S$.

\item[A6)] Let $S$ be
    $$\forall \sx\,(A(\sx, \sr)\to B(\sx, \sr)) \land  \forall \sx\,(B(\sx, \sr)\to C(\sx, \sr))
\Rightarrow \forall\sx\,(A(\sx, \sr)\to C(\sx, \sr))$$
and $|\sx| = n$.
It follows from (\cek), (\cb), (\cpp), and (\csmn) that
there exists a $V$-function $\kk$ such that, for all $b,\ c \in \II_{n+1}$, we have $\kk(b, c) \in \II_{n+1}$ and
\begin{equation}\label{fta6f}
  \varphi^V_{\mathsf{k}(b,c)}(m_1, \ldots, m_n, a) \simeq\varphi^V_c(m_1, \ldots, m_n, \varphi^V_b(m_1, \ldots, m_n, a))
\end{equation}
for all natural numbers $m_1, \ldots, m_n, a$.
By (\cek), (\cb), (\cfp), and (\cpp),

there is a natural number $e$ such that, for all natural numbers $k_1, \ldots, k_l, d$, we have
\begin{equation}\label{fta6n}
  \varphi^V_e(k_1, \ldots, k_l, d) \simeq \kk(\pp_1 d, \pp_2 d).
\end{equation}
Let $\varnothing \not= M \subseteq \NN$ and $f$ be an $M$-evaluation.
Suppose
\begin{equation}\label{fta6r}
  d \rvf \forall \sx\,(A(\sx, \sk) \to B(\sx, \sk))\land  \forall \sx\,(B(\sx, \sk)\to C(\sx, \sk))
\end{equation}
for some natural number $d$ and $\sk = k_1, \ldots, k_l \in M$.
Let us prove that
\begin{equation}\label{fta6e}
  \varphi^V_e(k_1, \ldots, k_l, d) \rvf \forall \sx\,(A(\sx, \sk)\to C(\sx, \sk)).
\end{equation}
From \eqref{fta6r} it follows that
\begin{equation}\label{fta6r1}
  \pp_1 d \rvf \forall \sx\,(A(\sx, \sk) \to B(\sx, \sk)),
\end{equation}
\begin{equation}\label{fta6r2}
  \pp_2 d \rvf \forall \sx\,(B(\sx, \sk) \to C(\sx, \sk)).
\end{equation}
Suppose
\begin{equation}\label{fta6g}
  a \rvf A(\sm, \sk)
\end{equation}
for some natural number $a$ and $\sm = m_1, \ldots, m_n \in M$.
Using \eqref{fta6r1}, \eqref{fta6g}, we obtain
\begin{equation}\label{fta6c}
  \varphi^V_{\pp_1 d}(m_1, \ldots, m_n, a) \rvf B(\sm, \sk).
\end{equation}
From \eqref{fta6c}, \eqref{fta6r2} it follows that
\begin{equation}\label{fta6c0}
  \varphi^V_{\pp_2 d}(m_1, \ldots, m_n, \varphi^V_{\pp_1 d}(m_1, \ldots, m_n, a)) \rvf C(\sm, \sk).
\end{equation}
Using \eqref{fta6f} and \eqref{fta6c0}, we get
\begin{equation}\label{fta6c00}
  \varphi^V_{\kk(\pp_1 d, \pp_2 d)}(m_1, \ldots, m_n, a) \rvf C(\sm, \sk).
\end{equation}
Thus
for all natural numbers $a$ and $m_1, \ldots, m_n \in M$
we have \eqref{fta6c00} whenever \eqref{fta6g}.
Hence
\begin{equation}\label{fta6s}
  \kk(\pp_1 d, \pp_2 d) \rvf \forall \sx\,(A(\sx, \sk)\to C(\sx, \sk)).
\end{equation}
Using \eqref{fta6n} and \eqref{fta6s}, we obtain \eqref{fta6e}.
Thus for all natural numbers $d$ and $k_1, \ldots, k_l \in M$ it follows from \eqref{fta6r} that \eqref{fta6e}.
Hence $e \rvfx{\sr} S$.

  \item[A7)] Let $S$ be
$$\forall \overline x\,(A(\sx, \sr)\to B(\sx, \sr))\land  \forall \overline x\,(A(\sx, \sr)\to C(\sx, \sr))
\Rightarrow \forall \overline x\,(A(\sx, \sr)\to B(\sx, \sr)\land  C(\sx, \sr))$$
and $|\sx| = n$.
It follows from (\cek), (\cb), and (\csmn) that
there exists a $V$-function $\kk$ such that, for all $b,\ c \in \II_{n+1}$, we have
$\kk(b, c) \in \II_{n+1}$ and
\begin{equation}\label{fta7f}
  \varphi^V_{\mathsf{k}(b,c)}(m_1, \ldots, m_n, a) \simeq
\mathsf{c}(\varphi^V_b(m_1, \ldots, m_n, a), \varphi^V_c(m_1, \ldots, m_n, a))
\end{equation}
for all natural numbers $m_1, \ldots, m_n, a$.
By (\cek), (\cb), (\cfp), and (\cpp)
 there is a natural number $e$ such that, for all natural numbers $k_1, \ldots, k_l, d$, we have
\begin{equation}\label{fta7n}
  \varphi^V_e(k_1, \ldots, k_l, d) \simeq \kk(\pp_1 d, \pp_2 d).
\end{equation}
Let $\varnothing \not= M \subseteq \NN$ and $f$ be an $M$-evaluation.
Suppose for some natural numbers $d$ and $k_1, \ldots, k_l \in M$,
\begin{equation}\label{fta7r}
  d \rvf \forall \sx\,(A(\sx, \sk) \to B(\sx, \sk))\land  \forall \sx\,(A(\sx, \sk)\to C(\sx, \sk)),
\end{equation}
where $\sk = k_1, \ldots, k_l$.
Let us prove that
\begin{equation}\label{fta7e}
  \varphi^V_e(k_1, \ldots, k_l, d) \rvf \forall \sx\,(A(\sx, \sk)\to B(\sx, \sk) \land C(\sx, \sk)).
\end{equation}
From \eqref{fta7r} it follows that
\begin{equation}\label{fta7r1}
  \pp_1 d \rvf \forall \sx\,(A(\sx, \sk) \to B(\sx, \sk)),
\end{equation}
\begin{equation}\label{fta7r2}
  \pp_2 d \rvf \forall \sx\,(A(\sx, \sk) \to C(\sx, \sk)).
\end{equation}
Suppose for some natural numbers $a$ and $m_1, \ldots, m_n \in M$,
\begin{equation}\label{fta7g}
  a \rvf A(\sm, \sk),
\end{equation}
where $\sm = m_1, \ldots, m_n$.
From \eqref{fta7r1}, \eqref{fta7g} it follows that
\begin{equation}\label{fta7c1}
  \varphi^V_{\pp_1 d}(m_1, \ldots, m_n, a) \rvf B(\sm, \sk),
\end{equation}
Using \eqref{fta7r2} and \eqref{fta7g}, we get
\begin{equation}\label{fta7c2}
  \varphi^V_{\pp_2 d}(m_1, \ldots, m_n, a) \rvf C(\sm, \sk).
\end{equation}
From \eqref{fta7c1}, \eqref{fta7c2} it follows that
\begin{equation}\label{fta7u}
  \cc(\varphi^V_{\pp_1 d}(m_1, \ldots, m_n, a), \varphi^V_{\pp_2 d}(m_1, \ldots, m_n, a))
  \rvf B(\sm, \sk) \land C(\sm, \sk).
\end{equation}
Using \eqref{fta7f} and \eqref{fta7u}, we obtain
\begin{equation}\label{fta7uf}
  \varphi^V_{\kk(\pp_1 d, \pp_2 d)}(m_1, \ldots, m_n, a) \rvf B(\sm, \sk) \land C(\sm, \sk).
\end{equation}
Thus for all natural numbers $a$ and $m_1, \ldots, m_n \in M$
we have \eqref{fta7uf} whenever \eqref{fta7g}.
Hence
\begin{equation}\label{fta7s}
  \kk(\pp_1 d, \pp_2 d) \rvf \forall \sx\,(A(\sx, \sk)\to B(\sx, \sk) \land C(\sx, \sk)).
\end{equation}
From \eqref{fta7n}, \eqref{fta7s} it follows that \eqref{fta7e}.
Thus for all natural numbers $d$ and $k_1, \ldots, k_l \in M$ it follows from \eqref{fta7r}
that \eqref{fta7e}. Hence
$
 e \rvfx{\sr} S.
$

  \item[A8)] Let $S$ be
$$\forall \overline x\,(B(\sx, \sr)\to A(\sx, \sr))\land \forall \overline x\,(C(\sx, \sr)\to A(\sx, \sr))
\Rightarrow \forall \overline x\,(B(\sx, \sr)\lor C(\sx, \sr)\to A(\sx, \sr))$$
and $|\sx| = n$.
It follows from (\ceu), (\cb), (\cek), and (\csmn)
that there exists a $V$-function $\kk$ such that
if $b, c \in \II_{n+1}$, then
$\kk(b, c) \in \II_{n+1}$ and
for all natural numbers $m_1, \ldots, m_n, a$ we have
\begin{align}\label{fta8f_if}
\varphi^V_{\mathsf{k}(b, c)}(m_1, \ldots, m_n, a) &\simeq
    \varphi^V_b(m_1, \ldots, m_n, \mathsf{p}_2 a)\ \mbox{ if } \pp_1 a = 0, \\
\label{fta8f_notif}
\varphi^V_{\mathsf{k}(b, c)}(m_1, \ldots, m_n, a) &\simeq
    \varphi^V_c(m_1, \ldots, m_n, \mathsf{p}_2 a)\ \mbox{ if } \pp_1 a \not= 0.
\end{align}
By (\cek), (\cb), (\cfp), and (\cpp)
there is a natural number $e$ such that, for all natural numbers $k_1, \ldots, k_l, d$, we have
\begin{equation}\label{fta8n}
  \varphi^V_e(k_1, \ldots, k_l, d) \simeq \kk(\pp_1 d, \pp_2 d).
\end{equation}
Let $\varnothing \not= M \subseteq \NN$ and $f$ be an $M$-evaluation.
Suppose for some natural numbers $d$ and $k_1, \ldots, k_l \in M$,
\begin{equation}\label{fta8r}
  d \rvf \forall \sx\,(B(\sx, \sk) \to A(\sx, \sk))\land  \forall \sx\,(C(\sx, \sk)\to A(\sx, \sk)),
\end{equation}
where $\sk = k_1, \ldots, k_l$.
Let us prove that
\begin{equation}\label{fta8e}
  \varphi^V_e(k_1, \ldots, k_l, d) \rvf \forall \sx\,(B(\sx, \sk) \lor C(\sx, \sk) \to A(\sx, \sk)).
\end{equation}
From \eqref{fta8r} it follows that
\begin{equation}\label{fta8r1}
  \pp_1 d \rvf \forall \sx\,(B(\sx, \sk) \to A(\sx, \sk)),
\end{equation}
\begin{equation}\label{fta8r2}
  \pp_2 d \rvf \forall \sx\,(C(\sx, \sk) \to A(\sx, \sk)).
\end{equation}
Suppose for some natural numbers $a$ and $m_1, \ldots, m_n \in M$,
\begin{equation}\label{fta8g}
  a \rvf B(\sm, \sk) \lor C(\sm, \sk),
\end{equation}
where $\sm = m_1, \ldots, m_n$.
From \eqref{fta8g} it follows that either $\pp_1 a = 0$, or $\pp_1 a = 1$.
Let us consider $2$ cases.

Case $1$: $\pp_1 a = 0$.
Then it follows from \eqref{fta8g} that
\begin{equation}\label{fta8g1}
  \pp_2 a \rvf B(\sm, \sk).
\end{equation}
Using \eqref{fta8r1} and \eqref{fta8g1}, we get
\begin{equation}\label{fta8c1}
  \varphi^V_{\pp_1 d}(m_1, \ldots, m_n, \pp_2 a) \rvf A(\sm, \sk),
\end{equation}
From \eqref{fta8f_if}, \eqref{fta8c1} it follows that
\begin{equation}\label{fta8f1}
  \varphi^V_{\kk(\pp_1 d, \pp_2 d)}(m_1, \ldots, m_n, a) \rvf A(\sm, \sk).
\end{equation}

Case $2$: $\pp_1 a = 1$.
Then it follows from \eqref{fta8g} that
\begin{equation}\label{fta8g2}
  \pp_2 a \rvf C(\sm, \sk).
\end{equation}
Using \eqref{fta8r2} and \eqref{fta8g2}, we obtain
\begin{equation}\label{fta8c2}
  \varphi^V_{\pp_2 d}(m_1, \ldots, m_n, \pp_2 a) \rvf A(\sm, \sk),
\end{equation}
From \eqref{fta8f_notif}, \eqref{fta8c2} it follows that \eqref{fta8f1}.

Thus for all natural numbers $a$ and $m_1, \ldots, m_n \in M$ we have \eqref{fta8f1} whenever \eqref{fta7g}.
Hence
\begin{equation}\label{fta8s}
  \kk(\pp_1 d, \pp_2 d) \rvf \forall \sx\,(B(\sx, \sk) \lor C(\sx, \sk) \to A(\sx, \sk)).
\end{equation}
From \eqref{fta8n}, \eqref{fta8s} it follows that \eqref{fta8e}.
Thus for all natural numbers $d$ and $k_1, \ldots, k_l \in M$
we have \eqref{fta8e} whenever \eqref{fta8r}. Hence $e \rvfx{\sr} S$.

  \item[A9)] Let $S$ be
$\forall \overline x\,(A \to B) \Rightarrow \forall \overline x\,([\sz / \sx]A \to [\sz/ \sx]B),$
where $|\sx| = |\sz| = n$.
Any variable in $\sz$ is in $\sx$ or in $\sr$.
We will write $\sz(\sx, \sr)$ instead of $\sz$.
Thus $S$ has the form
$$\forall \overline x\,(A(\sx, \sr) \to B(\sx, \sr))
\Rightarrow \forall \sx\,(A(\sz(\sx, \sr), \sr) \to B(\sz(\sx, \sr), \sr)).$$
For all natural numbers $\sk = k_1, \ldots, k_l$ denote by $\sz(\sx, \sk)$
the result of substituting $\sk$ for $\sr$ in $\sz(\sx, \sr)$.
If $\sm = m_1, \ldots, m_n$ is a list of natural numbers,
then denote by $\sz(\sm, \sk)$
the result of replacing $\sx$ by $\sm$ in $\sz(\sx, \sk)$.
Obviously, $\sz(\sm, \sk)$ is a list of natural numbers and $|\sz(\sm, \sk)| = n$.
For all $i = 1, \ldots, n$ denote by $\zz_i$ a function such that, for all natural numbers $\sm,\ \sk$,
$\zz_i(\sm, \sk)$ is the $i$-th element of $\sz(\sm, \sk)$.
Clearly, any $\zz_i$ is $I^j_{n+l}$ for some $j$.

It follows from (\cek), (\cb), (\cfp), (\cpp), and (\csmn) that there exists a $V$-function $\kk$ such that, for all $d \in \II_{n+1}$, we have $\kk(d) \in \II_{n+l+1}$ and
for all natural numbers $m_1, \ldots, m_n, a, k_1, \ldots, k_l$,
\begin{equation}\label{fta9f0}
  \varphi^V_{\kk(d)}(m_1, \ldots, m_n, a, k_1, \ldots, k_l) \simeq
  \varphi^V_d(\zz_1(\sm, \sk), \ldots, \zz_n(\sm, \sk), a),
\end{equation}
where $\sm = m_1, \ldots, m_n$,\ $\sk = k_1, \ldots, k_l$.
Since the list $\zz_1(\sm, \sk), \ldots, \zz_n(\sm, \sk)$ is $\sz(\sm, \sk)$,
we see that
\begin{equation}\label{fta9f00}
  \varphi^V_{\kk(d)}(m_1, \ldots, m_n, a, k_1, \ldots, k_l) \simeq
  \varphi^V_d(\sz(\sm, \sk), a).
\end{equation}
It follows from (\csmn) that
there exists a $V$-function $\ss$ such that, for all natural numbers $k_1, \ldots, k_l$ and $c \in \II_{n+l+1}$, we have
$\ss(c, k_1, \ldots, k_l) \in \II_{n+1}$ and
\begin{equation}\label{fta9f1}
  \varphi^V_{\ss(c, k_1, \ldots, k_l)}(m_1, \ldots, m_n, a)
  \simeq
  \varphi^V_{c}(m_1, \ldots, m_n, a, k_1, \ldots, k_l)
\end{equation}
for all natural numbers $m_1, \ldots, m_n, a$.
Using \eqref{fta9f00} and \eqref{fta9f1}, we get
\begin{equation}\label{fta9f}
  \varphi^V_{\ss(\kk(d), k_1, \ldots, k_l)}(m_1, \ldots, m_n, a) \simeq
  \varphi^V_d(\sz(\sm, \sk), a).
\end{equation}
It follows from (\cek), (\cfp), (\cpp), (\cb)
that there is a natural number $e$ such that, for all natural numbers $k_1, \ldots, k_l, d$, we have
\begin{equation}\label{fta9n}
  \varphi^V_e(k_1, \ldots, k_l, d) \simeq \ss(\kk(d), k_1, \ldots, k_l)
\end{equation}
Let $\varnothing \not= M \subseteq \NN$ and $f$ be an $M$-evaluation.
Suppose for some natural numbers $d$ and $k_1, \ldots, k_l \in M$,
\begin{equation}\label{fta9r}
  d \rvf \forall \sx\,(A(\sx, \sk) \to B(\sx, \sk)),
\end{equation}
where $\sk = k_1, \ldots, k_l$.
Let us prove that
\begin{equation}\label{fta9e}
  \varphi^V_e(k_1, \ldots, k_l, d) \rvf \forall \sx\,(A(\sz(\sx, \sk), \sk) \to B(\sz(\sx, \sk), \sk)).
\end{equation}
Suppose for some natural numbers $a$ and $m_1, \ldots, m_n \in M$,
\begin{equation}\label{fta9g}
  a \rvf A(\sz(\sm, \sk), \sk),
\end{equation}
where $\sm = m_1, \ldots, m_n$.
Using \eqref{fta9r} and \eqref{fta9g}, we obtain
\begin{equation}\label{fta9c}
  \varphi^V_d(\sz(\sm, \sk), a) \rvf B(\sz(\sm, \sk), \sk).
\end{equation}
From \eqref{fta9f}, \eqref{fta9c} it follows that
\begin{equation}\label{fta9s}
  \varphi^V_{\ss(\kk(d), k_1, \ldots, k_l)}(m_1, \ldots, m_n, a)
  \rvf B(\sz(\sm, \sk), \sk).
\end{equation}
Thus for all natural numbers $a$ and $m_1, \ldots, m_n \in M$
we have \eqref{fta9s} whenever \eqref{fta9g}.
Hence
\begin{equation}\label{fta9ss}
  \ss(\kk(d), k_1, \ldots, k_l) \rvf \forall \sx\,(A(\sz(\sx, \sk), \sk) \to B(\sz(\sx, \sk), \sk)).
\end{equation}
From \eqref{fta9n}, \eqref{fta9ss} it follows that \eqref{fta9e}.
Thus for all natural numbers $d$ and $k_1, \ldots, k_l \in M$
we have \eqref{fta9e} whenever \eqref{fta9r}. Hence $e \rvfx{\sr} S$.

  \item[A10)] Let $S$ be
$\forall \sx\,(A \to B) \Rightarrow \forall \sy\,(A \to B),$
where $\sx = x_1, \ldots, x_n$, $\sy = y_1, \ldots, y_p$ and
no variable in $\sy$ is free in $\forall \sx\,(A \to B)$.
Denote by $\su(\sr)$ a list of distinct variables that consists all free variables of $\forall \sx\,(A \to B)$.
For all natural numbers $\sk = k_1, \ldots, k_l$
denote by $\su(\sk)$ the result of replacing $\sr$ by $\sk$ in $\su(\sr)$.
Any variable in $\sx$ is in $\sy$ or in $\sr$.
We will write $\sx(\sy, \sr)$ instead of $\sx$.
For all natural numbers $\sk = k_1, \ldots, k_l$ denote by $\sx(\sy, \sk)$
the result of substituting $\sk$ for $\sr$ in $\sx(\sy, \sr)$.
If $\sm = m_1, \ldots, m_n$ is a list of natural numbers,
then denote by $\sx(\sm, \sk)$
the result of replacing $\sy$ by $\sm$ in $\sx(\sx, \sk)$.
Obviously, $\sx(\sm, \sk)$ is a list of natural numbers and $|\sx(\sm, \sk)| = n$.
For all $i = 1, \ldots, n$ denote by $\xx_i$ a function such that
$\xx_i(\sm, \sk)$ is the $i$-th element of $\sx(\sm, \sk)$
for all natural numbers $\sm,\ \sk$.
Clearly, any $\xx_i$ is $I^j_{n+l}$ for some $j$.
Thus $S$ has the form
$$ \forall \sx\,(A(\sx, \su(\sr)) \to B(\sx, \su(\sr))) \Rightarrow
\forall \sy\,(A(\sx(\sy, \sr), \su(\sr)) \to B(\sx(\sy, \sr), \su(\sr))).$$
It follows from (\cek), (\cb), (\csmn) that there exists
a $V$-function $\kk$ such that, for all $d \in \II_{n+1}$, we have $\kk(d) \in \II_{p+l+1}$ and
for all natural numbers $m_1, \ldots, m_p, a, k_1, \ldots, k_l$,
\begin{equation}\label{fta10f0}
  \varphi^V_{\kk(d)}(m_1, \ldots, m_p, a, k_1, \ldots, k_l) \simeq
  \varphi^V_d(\xx_1(\sm, \sk), \ldots, \xx_n(\sm, \sk), a),
\end{equation}
where $\sm = m_1, \ldots, m_p$,\ $\sk = k_1, \ldots, k_l$.
By (\csmn),  there is
a $V$-function $\ss$ such that, for all natural numbers $k_1, \ldots, k_l$ and $c \in \II_{p+l+1}$, we have $\ss(c, k_1, \ldots, k_l) \in \II_{p+1}$
and
\begin{equation}\label{fta10f1}
  \varphi^V_{\ss(c, k_1, \ldots, k_l)}(m_1, \ldots, m_p, a)
  \simeq
  \varphi^V_{c}(m_1, \ldots, m_p, a, k_1, \ldots, k_l)
\end{equation}
for all natural numbers $m_1, \ldots, m_p, a$.
Since the list $\xx_1(\sm, \sk), \ldots, \xx_n(\sm, \sk)$ is $\sx(\sm, \sk)$,
it follows from \eqref{fta10f0}, \eqref{fta10f1} that
for all natural numbers $m_1, \ldots, m_p, a, k_1, \ldots, k_l$ and $d \in \II_{n+1}$,
\begin{equation}\label{fta10f}
  \varphi^V_{\ss(\kk(d), k_1, \ldots, k_l)}(m_1, \ldots, m_p, a) \simeq
  \varphi^V_d(\sx(\sm, \sk), a).
\end{equation}
By (\cek), (\cfp), (\cb), (\cpp),
there is a natural number $e$ such that, for all natural numbers $k_1, \ldots, k_l, d$, we have
\begin{equation}\label{fta10n}
  \varphi^V_e(k_1, \ldots, k_l, d) \simeq \ss(\kk(d), k_1, \ldots, k_l).
\end{equation}
Let $\varnothing \not= M \subseteq \NN$ and $f$ be an $M$-evaluation.
Suppose for some natural numbers $d$ and $k_1, \ldots, k_l \in M$,
\begin{equation}\label{fta10r}
  d \rvf \forall \sx\,(A(\sx, \su(\sk)) \to B(\sx, \su(\sk))),
\end{equation}
where $\sk = k_1, \ldots, k_l$.
Let us prove that
\begin{equation}\label{fta10e}
  \varphi^V_e(k_1, \ldots, k_l, d) \rvf \forall \sy\,(A(\sx(\sy, \sk), \su(\sk)) \to B(\sx(\sy, \sk), \su(\sk))).
\end{equation}
Suppose for some natural numbers $a$ and $m_1, \ldots, m_p \in M$,
\begin{equation}\label{fta10g}
  a \rvf A(\sx(\sm, \sk), \su(\sk)),
\end{equation}
where $\sm = m_1, \ldots, m_p$.
Using \eqref{fta10r} and \eqref{fta10g}, we get
\begin{equation}\label{fta10c}
  \varphi^V_d(\sx(\sm, \sk), a) \rvf B(\sx(\sm, \sk), \sk)
\end{equation}
From \eqref{fta10f}, \eqref{fta10c} it follows that
\begin{equation}\label{fta10s}
  \varphi^V_{\ss(\kk(d), k_1, \ldots, k_l)}(m_1, \ldots, m_p, a)
  \rvf B(\sx(\sm, \sk), \sk)
\end{equation}
Thus for all natural numbers $a$ and $m_1, \ldots, m_p \in M$
we have \eqref{fta10s} whenever \eqref{fta10g}.
Hence
\begin{equation}\label{fta10ss}
  \ss(\kk(d), k_1, \ldots, k_l) \rvf \forall \sy\,(A(\sx(\sy, \sk), \su(\sk)) \to B(\sx(\sy, \sk), \su(\sk))).
\end{equation}
From \eqref{fta10n}, \eqref{fta10ss} it follows that \eqref{fta10e}.
Thus for all natural numbers $d$ and $k_1, \ldots, k_l \in M$
we have \eqref{fta10e} whenever \eqref{fta10r}. Hence $e \rvfx{\sr} S$.

  \item[A11)] Let $S$ be
$\forall \sx, x\,(B(\sx, x, \sr) \to A(\sx, \sr)) \Rightarrow
\forall \sx\,(\exists x\, B(\sx, x, \sr) \to A(\sx, \sr))$
and $|\sx| = n$.
It follows from (\cek), (\cb), (\cfp), (\cpp) that
there is a $V$-function $\kk$ such that, for every $d \in \II_{n+2}$, we have $\kk(d) \in \II_{n+1}$ and
\begin{equation}\label{fta11f}
  \varphi^V_{\kk(d)}(m_1, \ldots, m_n, b) \simeq \varphi^V_d(m_1, \ldots, m_n,\pp_1 b,\pp_2 b)
\end{equation}
for all natural numbers $m_1, \ldots, m_n, b$.
By (\cfp) and (\cpp),  there is a natural number $e$ such that, for all natural numbers $k_1, \ldots, k_l, d$,
\begin{equation}\label{fta11n}
  \varphi^V_e(k_1, \ldots, k_l, d) \simeq \kk(d).
\end{equation}
Let $\varnothing \not= M \subseteq \NN$ and $f$ be an $M$-evaluation.
Suppose for some natural numbers $d$ and $k_1, \ldots, k_l \in M$,
\begin{equation}\label{fta11r}
  d \rvf \forall \sx, x\,(B(\sx, x, \sk) \to A(\sx, \sk)),
\end{equation}
where $\sk = k_1, \ldots, k_l$.
Let us prove that
\begin{equation}\label{fta11e}
  \varphi^V_e(k_1, \ldots, k_l, d) \rvf \forall \sx\,(\exists x\, B(\sx, x, \sk) \to A(\sx, \sk)).
\end{equation}
Suppose for some natural numbers $b$ and $m_1, \ldots, m_n \in M$,
\begin{equation}\label{fta11g}
  b \rvf \exists x\, B(\sm, x, \sk),
\end{equation}
where $\sm = m_1, \ldots, m_n$.
From \eqref{fta11g} it follows that
\begin{equation}\label{fta11g1}
  \pp_2 b \rvf B(\sm, \pp_1 b, \sk).
\end{equation}
Using \eqref{fta11r} and \eqref{fta11g1}, we get
\begin{equation}\label{fta11c}
  \varphi^V_d(m_1, \ldots, m_n, \pp_1 b, \pp_2 b) \rvf A(\sm, \sk).
\end{equation}
From \eqref{fta11f}, \eqref{fta11c} it follows that
\begin{equation}\label{fta11s}
  \varphi^V_{\kk(d)}(m_1, \ldots, m_n, b)
  \rvf A(\sm, \sk).
\end{equation}
Thus for all natural numbers $b$ and $m_1, \ldots, m_n \in M$ we have \eqref{fta11s} whenever \eqref{fta11g}.
Hence
\begin{equation}\label{fta11ss}
  \kk(d) \rvf \forall \sx\,(\exists x\, B(\sx, x, \sk) \to A(\sx, \sk)).
\end{equation}
From \eqref{fta11n}, \eqref{fta11ss} it follows that \eqref{fta11e}.
Thus for all natural numbers $d$ and $k_1, \ldots, k_l \in M$
we have \eqref{fta11e} whenever \eqref{fta11r}. Hence $e \rvfx{\sr} S$.
\end{itemize}

Suppose $S$ is obtained by a rule of $\BQC$.
\begin{itemize}
  \item[R1)] Let $S$ be obtained by
$\frac{\displaystyle A\Rightarrow B\;\ B\Rightarrow C}{\displaystyle A\Rightarrow C}$
and $\su = u_1, \ldots, u_p$ be an admissible list of variables for
$A\Rightarrow B$, $B\Rightarrow C$, and $A\Rightarrow C$.
By the induction hypothesis, there exist natural numbers $a$, $b$ such that
\begin{equation}\label{ftr1i1}
  a \rvfx{\su} A\Rightarrow B,
\end{equation}
\begin{equation}\label{ftr1i2}
  b \rvfx{\su} B\Rightarrow C
\end{equation}
for every evaluation~$f$.
Using \eqref{ftr1i1} and \eqref{ftr1i2},  we get $a,\ b \in \II_{p+1}$.
It follows from (\cek), (\cb) that there is a natural number $c$ such that, for all natural numbers $k_1, \ldots, k_p, d$, we have
\begin{equation}\label{ftr1f}
  \varphi^V_c(k_1, \ldots, k_p, d) \simeq
  \varphi^V_b(k_1, \ldots, k_p, \varphi^V_a(k_1, \ldots, k_p, d)).
\end{equation}
Let $\varnothing \not= M \subseteq \NN$ and $f$ be an $M$-evaluation.
Suppose for some natural numbers $d$ and $k_1, \ldots, k_p \in M$,
\begin{equation}\label{ftr1g}
  d \rvf [\sk / \su]\,A,
\end{equation}
where $\sk = k_1, \ldots, k_p$.
From \eqref{ftr1i1}, \eqref{ftr1g} it follows that
\begin{equation}\label{ftr1c1}
  \varphi^V_a(k_1, \ldots, k_p, d) \rvf [\sk / \su]\,B.
\end{equation}
Using \eqref{ftr1i2} and \eqref{ftr1c1}, we get
\begin{equation}\label{ftr1c2}
  \varphi^V_b(k_1, \ldots, k_p, \varphi^V_a(k_1, \ldots, k_p, d)) \rvf [\sk / \su]\,C.
\end{equation}
From \eqref{ftr1f}, \eqref{ftr1c2} it follows that
\begin{equation}\label{ftr1c2n}
  \varphi^V_c(k_1, \ldots, k_p, d) \rvf [\sk / \su]\,C.
\end{equation}
Thus for all natural numbers $d$ and $k_1, \ldots, k_p \in M$ we have \eqref{ftr1c2n} whenever \eqref{ftr1g}.
Hence $c \rvfx{\su} A \Rightarrow C.$
Thus $c \rvfx{\su} S$ for all evaluations $f$.
It follows from Proposition \ref{p_eq} that there is a natural number $e$ such that
$e \rvfx{\sr} S$ for all evaluations $f$.

  \item[R2)] Let $S$ be obtained by
$\frac{\displaystyle A\Rightarrow B\;\ A\Rightarrow C}{\displaystyle A\Rightarrow B\land  C}$
and $\su = u_1, \ldots, u_p$ be an admissible list of variables for $A\Rightarrow B$, $A\Rightarrow C$, and $A \Rightarrow B \land C$.
By the induction hypothesis,
there exist natural numbers $b$, $c$ such that, for every evaluation~$f$,
\begin{equation}\label{ftr2i1}
  b \rvfx{\su} A\Rightarrow B,
\end{equation}
\begin{equation}\label{ftr2i2}
  c \rvfx{\su} A\Rightarrow C.
\end{equation}
It follows from \eqref{ftr2i1}, \eqref{ftr2i2} that  $b,\ c \in \II_{p+1}$.
By (\cek), there is a natural number $a$
such that, for all natural numbers $k_1, \ldots, k_p, d$,
\begin{equation}\label{ftr2f}
  \varphi^V_a(k_1, \ldots, k_p, d) \simeq
  \cc(\varphi^V_b(k_1, \ldots, k_p, d),\, \varphi^V_c(k_1, \ldots, k_p, d)).
\end{equation}
Let $\varnothing \not= M \subseteq \NN$ and $f$ be an $M$-evaluation.
Suppose for some natural numbers
$d$ and $k_1, \ldots, k_p \in M$,
\begin{equation}\label{ftr2g}
  d \rvf [\sk / \su]\,A,
\end{equation}
where $\sk = k_1, \ldots, k_p$.
From \eqref{ftr2i1}, \eqref{ftr2g} it follows that
\begin{equation}\label{ftr2c1}
  \varphi^V_b(k_1, \ldots, k_p, d) \rvf [\sk / \su]\,B.
\end{equation}
Using \eqref{ftr2i2} and \eqref{ftr2g}, we get
\begin{equation}\label{ftr2c2}
  \varphi^V_c(k_1, \ldots, k_p, d) \rvf [\sk / \su]\,C.
\end{equation}
From \eqref{ftr2c1}, \eqref{ftr2c2} it follows that
\begin{equation}\label{ftr2u}
  \cc(\varphi^V_b(k_1, \ldots, k_p, d),\, \varphi^V_c(k_1, \ldots, k_p, d)) \rvf [\sk / \su]\,(B \land C).
\end{equation}
Using \eqref{ftr2f} and \eqref{ftr2u}, we obtain
\begin{equation}\label{ftr2s}
  \varphi^V_a(k_1, \ldots, k_p, d) \rvf [\sk / \su]\,(B \land C).
\end{equation}
Thus for all natural numbers $d$ and $k_1, \ldots, k_p \in M$
we have \eqref{ftr2s} whenever \eqref{ftr2g}.
Hence
$
  a \rvfx{\su} A \Rightarrow B \land C.
$
Thus $a \rvfx{\su} S$ for all evaluations $f$.
It follows from Proposition~\ref{p_eq} that there is a natural number $e$ such that
$e \rvfx{\sr} S$ for all evaluations $f$.

  \item[R3) a)] Let $S$ be obtained by
$\frac{\displaystyle A\Rightarrow B \land  C}{\displaystyle A\Rightarrow B}$
and
$\su = u_1, \ldots, u_p$ be an admissible list of variables for $A\Rightarrow B$ and $A \Rightarrow B \land C$.
By the induction hypothesis,
there is a natural number $a$ such that, for every evaluation~$f$, we have
\begin{equation}\label{ftr3i}
  a \rvfx{\su} (A\Rightarrow B \land C).
\end{equation}
It follows from \eqref{ftr3i} that  $a \in \II_{p+1}$.
By (\cek), there is a natural number $b$
such that, for all natural numbers $k_1, \ldots, k_p, d$,
\begin{equation}\label{ftr3f1}
  \varphi^V_b(k_1, \ldots, k_p, d) \simeq
  \pp_1 \varphi^V_a(k_1, \ldots, k_p, d).
\end{equation}
Let $\varnothing \not= M \subseteq \NN$ and $f$ be an $M$-evaluation.
Suppose for some natural numbers
$d$ and $k_1, \ldots, k_p \in M$,
\begin{equation}\label{ftr3g}
  d \rvf [\sk / \su]\,A,
\end{equation}
where $\sk = k_1, \ldots, k_p$.
From \eqref{ftr3i}, \eqref{ftr3g} it follows that
\begin{equation}\label{ftr3c}
  \varphi^V_a(k_1, \ldots, k_p, d) \rvf [\sk / \su]\,(B \land C).
\end{equation}
Using \eqref{ftr3c}, we get
\begin{equation}\label{ftr3cp1}
  \pp_1 \varphi^V_a(k_1, \ldots, k_p, d) \rvf [\sk / \su]\,B.
\end{equation}
From \eqref{ftr3f1}, \eqref{ftr3cp1} it follows that
\begin{equation}\label{ftr3s1}
  \varphi^V_b(k_1, \ldots, k_p, d) \rvf [\sk / \su]\,B.
\end{equation}
Thus for all natural numbers $d$ and $k_1, \ldots, k_p \in M$ we have \eqref{ftr3s1} whenever \eqref{ftr3g}.
Hence
$
  b \rvfx{\su} A \Rightarrow B.
$
Thus $b \rvfx{\su} S$ for all evaluations $f$.
It follows from Proposition~\ref{p_eq} that there is a natural number $e$ such that
$e \rvfx{\sr} S$ for all evaluations $f$.

  \item[    b)] Let $S$ be obtained by
$\frac{\displaystyle A\Rightarrow B \land  C}{\displaystyle A\Rightarrow C}$
and $\su = u_1, \ldots, u_p$ be an admissible list of variables for $A\Rightarrow C$ and $A \Rightarrow B \land C$.
By the induction hypothesis,
there is a natural number $a$ such that, for every evaluation~$f$, we have \eqref{ftr3i}.
Obviously, $a \in \II_{p+1}$.
It follows from (\cek) that there is a natural number $b$
such that, for all natural numbers $k_1, \ldots, k_p, d$, we have
\begin{equation}\label{ftr3f2}
  \varphi^V_b(k_1, \ldots, k_p, d) \simeq
  \pp_2 \varphi^V_a(k_1, \ldots, k_p, d).
\end{equation}
It can be easily checked that $b \rvfx{\su} S$ for all evaluations $f$.
It follows from Proposition~\ref{p_eq} that there is a natural number $e$ such that
$e \rvfx{\sr} S$ for all evaluations $f$.

  \item[R4)] Let $S$ be obtained by
$\frac{\displaystyle B\Rightarrow A\;\ C\Rightarrow A}{\displaystyle B\lor C\Rightarrow A}$
and $\su = u_1, \ldots, u_p$ be an admissible list of variables for $B\Rightarrow A,\ C \Rightarrow A$, and $B \lor C \Rightarrow A$.
By the induction hypothesis,
there exist natural numbers $b$, $c$ such that, for every evaluation~$f$,
\begin{equation}\label{ftr4i1}
  b \rvfx{\su} B\Rightarrow A,
\end{equation}
\begin{equation}\label{ftr4i2}
  c \rvfx{\su} C\Rightarrow A.
\end{equation}
It follows from \eqref{ftr4i1}, \eqref{ftr4i2} that  $b,\ c \in \II_{p+1}$.
By (\ceu), (\cek), (\cb), there exists a natural number $a$ such that, for all natural numbers $k_1, \ldots, k_p, d$,
we have
\begin{align}\label{ftr4f_if}
\varphi^V_a(k_1, \ldots, k_p, d) \simeq
    \varphi^V_b(k_1, \ldots, k_p, \mathsf{p}_2 d)\ \mbox{ if } \pp_1 d = 0, \\
\label{ftr4f_notif}
\varphi^V_a(k_1, \ldots, k_p, d) \simeq
    \varphi^V_c(k_1, \ldots, k_p, \mathsf{p}_2 d)\ \mbox{ if } \pp_1 d \not= 0.
\end{align}
Let $\varnothing \not= M \subseteq \NN$ and $f$ be an $M$-evaluation.
Suppose for some natural numbers
$d$ and $k_1, \ldots, k_p \in M$,
\begin{equation}\label{ftr4g}
  d \rvf [\sk / \su]\,(B \lor C),
\end{equation}
where $\sk = k_1, \ldots, k_p$.
From \eqref{ftr4g} it follows that either $\pp_1 d = 0$, or $\pp_1 d = 1$.
Let us consider $2$ cases.

Case $1$: $\pp_1 d = 0$.
Using \eqref{ftr4g}, we get
\begin{equation}\label{ftr4g1}
  \pp_2 d \rvf [\sk / \su]\,B.
\end{equation}
From \eqref{ftr4i1}, \eqref{ftr4g1} it follows that
\begin{equation}\label{ftr4c1}
  \varphi^V_{b}(k_1, \ldots, k_p, \pp_2 d) \rvf [\sk / \su]\,A.
\end{equation}
Using \eqref{ftr4f_if} and \eqref{ftr4c1}, we obtain
\begin{equation}\label{ftr4f1}
  \varphi^V_{a}(k_1, \ldots, k_p, d) \rvf [\sk / \su]\,A.
\end{equation}

Case $2$: $\pp_1 d = 1$.
Using \eqref{ftr4g}, we get
\begin{equation}\label{ftr4g2}
  \pp_2 d \rvf [\sk / \su]\,C.
\end{equation}
From \eqref{ftr4i2}, \eqref{ftr4g2} it follows that
\begin{equation}\label{ftr4c2}
  \varphi^V_{c}(k_1, \ldots, k_p, \pp_2 d) \rvf [\sk / \su]\,A.
\end{equation}
Using \eqref{ftr4f_notif} and \eqref{ftr4c2}, we get \eqref{ftr4f1}.

Thus for all natural numbers $d$ and $k_1, \ldots, k_p \in M$ we have \eqref{ftr4f1} whenever \eqref{ftr4g}.
Hence
$
  a \rvfx{\su} B \lor C \Rightarrow A.
$
Thus $a \rvfx{\su} S$ for all evaluations $f$.
It follows from Proposition~\ref{p_eq} that there is a natural number $e$ such that
$e \rvfx{\sr} S$ for all evaluations $f$.

  \item[R5) a)] Let $S$ be obtained by
$\frac{\displaystyle B\lor C\Rightarrow A}{\displaystyle B\Rightarrow A}$
and $\su = u_1, \ldots, u_p$ be an admissible list of variables for $B\Rightarrow A$ and $B \lor C \Rightarrow A$.
By the induction hypothesis,
there is a natural number $a$ such that, for every evaluation~$f$, we have
\begin{equation}\label{ftr5i}
  a \rvfx{\su} B \lor C \Rightarrow A.
\end{equation}
It follows from \eqref{ftr5i} that  $a \in \II_{p+1}$.
By (\cek), (\cb), (\cec), there is a natural number $b$
such that, for all natural numbers $k_1, \ldots, k_p, d$,
\begin{equation}\label{ftr5f1}
  \varphi^V_b(k_1, \ldots, k_p, d) \simeq
  \varphi^V_a(k_1, \ldots, k_p, \cc(0, d)).
\end{equation}
Let $\varnothing \not= M \subseteq \NN$ and $f$ be an $M$-evaluation.
Suppose for some natural numbers
$d$ and $k_1, \ldots, k_p \in M$,
\begin{equation}\label{ftr5g1}
  d \rvf [\sk / \su]\,B,
\end{equation}
where $\sk = k_1, \ldots, k_p$.
From \eqref{ftr5g1} it follows that
\begin{equation}\label{ftr5g10}
  \cc(0, d) \rvf [\sk / \su]\,(B \lor C).
\end{equation}
Using \eqref{ftr5i} and \eqref{ftr5g10}, we get
\begin{equation}\label{ftr5c1}
  \varphi^V_a(k_1, \ldots, k_p, \cc(0, d)) \rvf [\sk / \su]\,A.
\end{equation}
From \eqref{ftr5f1}, \eqref{ftr5c1} it follows that
\begin{equation}\label{ftr5s}
  \varphi^V_b(k_1, \ldots, k_p, d) \rvf [\sk / \su]\,A.
\end{equation}
Thus for all natural numbers $d$ and $k_1, \ldots, k_p \in M$ we have \eqref{ftr5s} whenever \eqref{ftr5g1}.
Hence
$
  b \rvfx{\su} B \Rightarrow A.
$
Thus $b \rvfx{\su} S$ for all evaluations $f$.
It follows from Proposition~\ref{p_eq} that there is a natural number $e$ such that
$e \rvfx{\sr} S$ for all evaluations $f$.

  \item[    b)] Let $S$ be obtained by
$\frac{\displaystyle B\lor C\Rightarrow A}{\displaystyle C\Rightarrow A}$
and $\su = u_1, \ldots, u_p$ be an admissible list of variables for $C\Rightarrow A$ and $B \lor C \Rightarrow A$.
By the induction hypothesis,
there is a natural number $a$ such that, for every evaluation~$f$, we have \eqref{ftr5i}.
It follows from \eqref{ftr5i} that  $a \in \II_{p+1}$.
By (\cek), (\cb), (\cec), there is a natural number $b$ such that, for all natural numbers $k_1, \ldots, k_p, d$,
\begin{equation}\label{ftr5f2}
  \varphi^V_b(k_1, \ldots, k_p, d) \simeq
  \varphi^V_a(k_1, \ldots, k_p, \cc(1, d)).
\end{equation}
It can be easily checked that
$b \rvfx{\su} S$ for all evaluations $f$.
It follows from Proposition~\ref{p_eq} that there is a natural number $e$ such that
$e \rvfx{\sr} S$ for all evaluations $f$.

  \item[R6)] Let $S$ be obtained by
$\frac{\displaystyle A \Rightarrow B}{\displaystyle [\sy / \sx]\,A \Rightarrow [\sy / \sx]\,B}\ $
and $|\sx| = |\sy| = n$.
Suppose $\su = u_1, \ldots, u_p$ is an admissible list of variables for $A \Rightarrow B$ and
$[\sy / \sx]\,A \Rightarrow [\sy / \sx]\,B$.
By $[\sy / \sx]\,\su$ denote the result of substituting $\sy$ for $\sx$ in $\su$.
By $\su$ is admissible for $[\sy / \sx]\,A \Rightarrow [\sy / \sx]\,B$,
all variables in $[\sy / \sx] \su$ are in $\su$.
If $\sk = k_1, \ldots, k_p$ is a list of natural numbers,
then by $[\sk / \su][\sy / \sx]\,\su$ denote the result of substituting $\sk$ for $\su$ in $[\sy / \sx]\,\su$.
Obviously, $[\sk / \su][\sy / \sx]\,\su$ is a list of natural numbers and $|[\sk / \su][\sy / \sx]\,\su| = p$.
For all $i = 1, \ldots, l$ by $\zz_i$ denote a function such that
$\zz_i(k_1, \ldots, k_p)$ is the $i$-th element of $[\sk / \su][\sy / \sx]\,\su$
for all natural numbers $k_1, \ldots, k_p$.
Clearly, for all $i$ there exists $j$ such that $\zz_i$ is $\II^j_{l+1}$.
By the induction hypothesis,
there is a natural number $a$ such that, for every evaluation~$f$, we have
\begin{equation}\label{ftr6i}
  a \rvfx{\su} A \Rightarrow B.
\end{equation}
From \eqref{ftr6i} it follows that $a \in \II_{p+1}$.
By (\cek), (\cb), there is a natural number $b$ such that, for all natural numbers $d$ and $\sk= k_1, \ldots, k_p$,
\begin{equation}\label{ftr6f0}
  \varphi^V_b(k_1, \ldots, k_p, d) \simeq
  \varphi^V_a(\zz_1(\sk), \ldots, \zz_l(\sk), d).
\end{equation}
By $[\sk / \su][\sy / \sx]\,\su$ is $\zz_1(\sk), \ldots, \zz_l(\sk)$,

\begin{equation}\label{ftr6f}
  \varphi^V_b(k_1, \ldots, k_p, d) \simeq
  \varphi^V_a([\sk / \su][\sy / \sx]\,\su, d).
\end{equation}
Let $\varnothing \not= M \subseteq \NN$ and $f$ be an $M$-evaluation.
Suppose for some natural numbers $d$ and $k_1, \ldots, k_p \in M$,
\begin{equation}\label{ftr6g}
  d \rvf [\sk / \su][\sy / \sx]\,A,
\end{equation}
where $\sk = k_1, \ldots, k_p$.
From \eqref{ftr6i}, \eqref{ftr6g} it follows that
\begin{equation}\label{ftr6c}
  \varphi^V_a([\sk / \su][\sy / \sx]\,\su, d) \rvf [\sk / \su][\sy / \sx]\,B.
\end{equation}
Using \eqref{ftr6f} and \eqref{ftr6c}, we get
\begin{equation}\label{ftr6s}
  \varphi^V_b(k_1, \ldots, k_p, d) \rvf [\sk / \su][\sy / \sx]\,B.
\end{equation}
Thus for all natural numbers $d$ and $k_1, \ldots, k_p \in M$ we have \eqref{ftr6s} whenever \eqref{ftr6g}.
Hence
$
  b \rvfx{\su} [\sy / \sx]\,A \Rightarrow [\sy / \sx]\,B.
$
Thus $b \rvfx{\su} S$ for all evaluations $f$.
It follows from Proposition~\ref{p_eq} that there is a natural number $e$ such that
$e \rvfx{\sr} S$ for all evaluations $f$.

  \item[R7)] Let $S$ be obtained by
$\frac{\displaystyle B\Rightarrow A}{\displaystyle \exists x\, B\Rightarrow A},$
where $x$ is not free in $A$.
It is clear that $S$ has the form $\exists x\, B(\su, x) \Rightarrow A(\su)$
for some list of variables $\su = u_1, \ldots, u_p$.
By the induction hypothesis, there is a natural number $a$ such that, for every evaluation~$f$, we have
\begin{equation}\label{ftr7i}
  a \rvfx{\su, x} B \Rightarrow A.
\end{equation}
It follows from \eqref{ftr7i} that  $a \in \II_{l+2}$.
By (\cek), (\cb), there is a natural number $b$ such that, for all natural numbers $k_1, \ldots, k_p, d$,
\begin{equation}\label{ftr7f}
  \varphi^V_b(k_1, \ldots, k_p, d) \simeq
  \varphi^V_a(k_1, \ldots, k_p, \pp_1 d, \pp_2 d).
\end{equation}
Let $\varnothing \not= M \subseteq \NN$ and $f$ be an $M$-evaluation.
Suppose for some natural numbers
$d$ and $k_1, \ldots, k_p \in M$,
\begin{equation}\label{ftr7g}
  d \rvf \exists x\, B(\sk, x),
\end{equation}
where $\sk = k_1, \ldots, k_p$.
From \eqref{ftr7g} it follows that
\begin{equation}\label{ftr7g0}
  \pp_2 d \rvf B(\sk, \pp_1 d).
\end{equation}
Using \eqref{ftr7i} and \eqref{ftr7g0}, we get
\begin{equation}\label{ftr7c}
  \varphi^V_a(k_1, \ldots, k_p, \pp_1 d, \pp_2 d) \rvf A(\sk).
\end{equation}
From \eqref{ftr7f}, \eqref{ftr7c} it follows that
\begin{equation}\label{ftr7s}
  \varphi^V_b(k_1, \ldots, k_p, d) \rvf A(\sk).
\end{equation}
Thus for all natural numbers $d$ and $k_1, \ldots, k_p \in M$ we have \eqref{ftr7s} whenever \eqref{ftr7g}.
Hence
$
  b \rvfx{\su} \exists x\,B \Rightarrow A.
$
Thus $b \rvfx{\su} S$ for all evaluations $f$.
It follows from Proposition~\ref{p_eq} that there is a natural number $e$ such that
$e \rvfx{\sr} S$ for all evaluations $f$.

  \item[R8)] Let $S$ be obtained by
$\frac{\displaystyle \exists x\,B\Rightarrow A}{\displaystyle B\Rightarrow A},$
where $x$ is not free in $A$.
It is clear that $S$ has the form $B(\su, x) \Rightarrow A(\su)$
for some list of variables $\su = u_1, \ldots, u_p$.
By the induction hypothesis, there is a natural number $a$ such that, for every evaluation~$f$, we have
\begin{equation}\label{ftr8i}
  a \rvfx{\su} \exists x\,B \Rightarrow A.
\end{equation}
It follows from \eqref{ftr8i} that $a \in \II_{p+1}$.
By (\cek), (\cb), there is a natural number $b$ such that, for all natural numbers $k_1, \ldots, k_p, c, d$,
\begin{equation}\label{ftr8f}
  \varphi^V_b(k_1, \ldots, k_p, c, d) \simeq
  \varphi^V_a(k_1, \ldots, k_p, \cc(c, d)).
\end{equation}
Let $\varnothing \not= M \subseteq \NN$ and $f$ be an $M$-evaluation.
Suppose for some natural numbers
$d$ and $k_1, \ldots, k_p, c \in M$,
\begin{equation}\label{ftr8g}
  d \rvf B(\sk, c),
\end{equation}
where $\sk = k_1, \ldots, k_p$.
From \eqref{ftr8g} it follows that
\begin{equation}\label{ftr8g0}
  \cc(c, d) \rvf \exists x\, B(\sk, x).
\end{equation}
Using \eqref{ftr8i} and \eqref{ftr8g0}, we get
\begin{equation}\label{ftr8c}
  \varphi^V_a(k_1, \ldots, k_p, \cc(c, d)) \rvf A(\sk).
\end{equation}
From \eqref{ftr8f}, \eqref{ftr8c} it follows that
\begin{equation}\label{ftr8s}
  \varphi^V_b(k_1, \ldots, k_p, c, d) \rvf A(\sk).
\end{equation}
Thus for all natural numbers $d$ and $k_1, \ldots, k_p, c \in M$ we have \eqref{ftr8s} whenever \eqref{ftr8g}.
Hence
$
  b \rvfx{\su, x} B \Rightarrow A.
$
Thus $b \rvfx{\su, x} S$ for all evaluations $f$.
It follows from Proposition~\ref{p_eq} that there is a natural number $e$ such that
$e \rvfx{\sr} S$ for all evaluations $f$.

  \item[R9)] Let $S$ be obtained by
$\frac{\displaystyle A \land  B\Rightarrow C}{\displaystyle A \Rightarrow \forall \sx\,(B\to C)},$
where $|\sx| = n$ and all variables in $\sx$ are not free in $A$.
It is clear that
$S$ has the form
$$A(\su) \Rightarrow \forall \sx\,(B(\sx, \su) \to C(\sx, \su))$$
for some list of variables $\su = u_1, \ldots, u_p$.
By the induction hypothesis, there is a natural number $c$ such that, for every evaluation~$f$,
\begin{equation}\label{ftr9i}
  c \rvfx{\sx, \su} A(\su) \land  B(\sx, \su) \Rightarrow C(\sx, \su).
\end{equation}
It follows from \eqref{ftr9i} that  $c \in \II_{n+l+1}$.
By (\cek), (\cb), (\csmn), (\cpp),
there exists a $V$-function $\ss$ such that we have
\begin{equation}\label{ftr9f}
  \varphi^V_{\ss(c, k_1, \ldots, k_p, d)}(m_1, \ldots, m_n, b) \simeq
\varphi^V_c(m_1, \ldots, m_n, k_1, \ldots, k_p, \cc(d, b))
\end{equation}
for all natural numbers $m_1, \ldots, m_n, k_1, \ldots, k_p, d, b$.
It follows from (\cpp), (\csmn) that there is a natural number $e$ such that, for all natural numbers $k_1, \ldots, k_p, d$, we have
\begin{equation}\label{ftr9n}
  \varphi^V_e(k_1, \ldots, k_p, d) \simeq \ss(c, k_1, \ldots, k_p, d).
\end{equation}
Let $\varnothing \not= M \subseteq \NN$ and $f$ be an $M$-evaluation.
Suppose for some natural numbers
$d$ and $k_1, \ldots, k_p \in M$,
\begin{equation}\label{ftr9r}
  d \rvf A(\sk),
\end{equation}
where $\sk = k_1, \ldots, k_p$.
Let us prove that
\begin{equation}\label{ftr9e}
  \varphi^V_e(k_1, \ldots, k_p, d) \rvf \forall \sx\,(B(\sx, \sk) \to C(\sx, \sk)).
\end{equation}
Suppose for some natural numbers $b$ and $m_1, \ldots, m_n \in M$,
\begin{equation}\label{ftr9g}
  b \rvf B(\sm, \sk),
\end{equation}
where $\sm = m_1, \ldots, m_n$.
From \eqref{ftr9r}, \eqref{ftr9g} it follows that
\begin{equation}\label{ftr9u}
  \cc(d, b) \rvf A(\sk) \land B(\sm, \sk).
\end{equation}
Using \eqref{ftr9i} and \eqref{ftr9u}, we get
\begin{equation}\label{ftr9c}
  \varphi^V_c(m_1, \ldots, m_n, k_1, \ldots, k_p, \cc(d, b)) \rvf C(\sm, \sk).
\end{equation}
From \eqref{ftr9f}, \eqref{ftr9c} it follows that
\begin{equation}\label{ftr9s}
  \varphi^V_{\ss(c, k_1, \ldots, k_p, d)}(m_1, \ldots, m_n, b)
  \rvf C(\sm, \sk).
\end{equation}
Thus for all natural numbers $b$ and $m_1, \ldots, m_n \in M$ we have \eqref{ftr9s} whenever \eqref{ftr9g}.
Hence
\begin{equation}\label{ftr9ss}
  \ss(c, k_1, \ldots, k_p, d) \rvf \forall \sx\,(B(\sx, \sk) \to C(\sx, \sk)).
\end{equation}
From \eqref{ftr9n}, \eqref{ftr9ss} it follows that \eqref{ftr9e}.
Thus for all natural numbers $d$ and $k_1, \ldots, k_p \in M$ we have \eqref{ftr9e} whenever \eqref{ftr9r}.
Hence
$$
  e \rvfx{\su}\,  A \Rightarrow \forall \sx\,(B\to C).
$$
Thus $e \rvfx{\su} S$ for all evaluations $f$.
It follows from Proposition~\ref{p_eq} that there exists a natural number $e'$ such that
$e' \rvfx{\sr} S$ for all evaluations $f$.
\end{itemize}
\end{proof}

\begin{theorem}
  If a sentence $A$ is derivable in $\BQC$,
  then the sentence $A$ is absolutely $V$-realizable over all domains.
\end{theorem}
\begin{proof}
  Let $A$ be derivable in $\BQC$. Then $\BQC \vdash \top \Rightarrow A$.
  Since $A$ is a sentence, we see that an empty list of variables $\overline{v}$ is admissible for $\top \Rightarrow A$.
  From Theorem~\ref{t_main} it follows that
  there exists a natural number $e$ such that $e \rvfx{\overline{v}} \top \Rightarrow A$ for all evaluations $f$.
  Then $e' \rvf A$ for all evaluations $f$, where $e' = \varphi^V_e(0)$.
\end{proof}

\section*{Acknowledgments}
This research was partially supported by Russian Foundation for Basic Research under grant 20-01-00670.

\end{document}